\documentclass[11pt,reqno,english]{extarticle}
\PassOptionsToPackage{backgroundcolor=green!40}{todonotes}
\usepackage[tbtags]{amsmath}
\usepackage{amsthm}
\usepackage{amssymb}
\usepackage{mathtools}
\usepackage{textcomp}
\usepackage[utf8]{inputenc}
\usepackage{csquotes}
\usepackage{xcolor}
\usepackage{babel}
\usepackage{verbatim}
\usepackage{mathrsfs}
\usepackage{mathtools}
\usepackage{url}
\usepackage{geometry}
\usepackage[pdfusetitle,
 bookmarks=true,bookmarksnumbered=false,bookmarksopen=false,
 breaklinks=true,pdfborder={0 0 0},pdfborderstyle={},backref=false,colorlinks=true]
 {hyperref}
\hypersetup{
 urlcolor=blue,linkcolor=blue,citecolor=blue}
\usepackage[style=trad-plain,url=false,isbn=false]{biblatex}
\usepackage[shortlabels]{enumitem}
\usepackage{aligned-overset}
\usepackage{cancel}
\newtheorem{theorem}{Theorem}[section]
\newtheorem{lemma}[theorem]{Lemma}
\newtheorem{proposition}[theorem]{Proposition}
\newtheorem{corollary}[theorem]{Corollary}
\theoremstyle{definition}
\newtheorem{definition}[theorem]{Definition}
\theoremstyle{remark}
\newtheorem{remark}[theorem]{Remark}
\numberwithin{equation}{section}

\usepackage{mleftright}
\mleftright
\newcommand{\email}[1]{{\href{mailto:#1}{\nolinkurl{#1}}}}

\DeclareMathOperator{\Lip}{Lip}

\DeclareMathOperator{\Cov}{Cov}
\newcommand{\eps}{\varepsilon}

\title{McKean--Vlasov limits of scaling-critical reaction-diffusion equations with random initial data}

\author{Bryan Castillo\thanks{Department of Mathematics, Duke University, Durham, NC, 27708, USA. \emph{Email addresses:} \email{bryan.castillo@duke.edu}, \email{alexander.dunlap@duke.edu}}\and Alexander Dunlap\footnotemark[1]}

\begin{document}
\maketitle
\begin{abstract}
  We study a large class of scaling-critical reaction-diffusion equations in two spatial dimensions, where the initial data is white noise mollified at scale $\eps^2$ and the reaction term is attenuated by a factor of $(\log\eps^{-1})^{-1}$. We show that as $\eps\to 0$, the solution converges to the solution of a McKean--Vlasov equation, which is Gaussian with standard deviation given by the solution to an ODE. Our result covers the case of the reaction term $f(u)=u^3$, and thus gives a new proof of the limiting behavior for the Allen--Cahn equation discovered in the recent work of Gabriel, Rosati, and Zygouras (\emph{Probab.\ Theory Related Fields}~192: 1373--1446, 2025).
\end{abstract}

\tableofcontents
	\section{Introduction}
    This paper is concerned with the study of a scaling-critical reaction-diffusion equation with white noise (random) initial data. The idealized equation of interest is for a scalar-valued function $u = u(t,x)$, for $t\in \mathbb{R}_{\geq 0},x\in \mathbb{R}^2$:
    \begin{equation}
    \label{eq:archetype}
        \begin{cases}
            \partial_t u = \frac{1}{2}\Delta u + \mathfrak{m}u - t^{-3/2}F(t^{1/2}u);\\
            u(0,x) = \eta(x),
        \end{cases}
    \end{equation}
    where $\eta$ is a spatial white noise on $\mathbb{R}^2$ (defined with respect to a probability space $(\Omega,\mathcal{F},\mathbb{P})$) and $F$ is the reaction term. A special case of great importance is when $F(u)=\lambda^2u^3$, in which case (\ref{eq:Allen-Cahn-archetype}) becomes the Allen--Cahn equation with white noise initial data:
    \begin{equation}
    \label{eq:Allen-Cahn-archetype}
      \begin{cases}
          \partial_t u=\frac12\Delta u+\mathfrak{m} u - \lambda^2 u^3          ;\\
          u(0,x)=\eta(x).
      \end{cases}
    \end{equation}
    Most notably, for $\mathfrak{m}>0$, the problem (\ref{eq:Allen-Cahn-archetype}) is connected to the study of mean-curvature flow, where many conjectured properties of the coarsening of sets are expected to hold for suitable random initial conditions \cites{BRAY199341,PhysRevLett.49.1223,MR4603398,article}. A related problem, which has been the subject of significant interest in the mathematical physics literature, is when instead of a random initial condition, the equation (\ref{eq:Allen-Cahn-archetype}) is augmented with additive space-time white noise forcing. Then, the (at least formal) invariant measure is the $\Phi^4$ field. We do not attempt to survey the vast literature on this topic, but see for example \cites{parisi:wu:1981:perturbation,MR3693966,MR2928722,MR4276286,duch2025ergodicityinfinitevolumephi43}.
    
    The term ``scaling-critical'' is most properly applied to the equation (\ref{eq:archetype}) with $\mathfrak{m} = 0$. In this case, if $u$ solves (\ref{eq:archetype}) with $\mathfrak{m}=0$ and we define $u^\theta(t,x) = \theta u(\theta^2 t,\theta x)$, then $u^\theta$ also solves (\ref{eq:archetype}). Scaling-criticality represents both an opportunity and a challenge in the study of stochastic systems -- an opportunity because of the possibilities it yields for multiscale analysis, but a challenge because typically scale-invariant systems are often not well-posed.
    Indeed, the Hölder--Besov regularity of $\eta$ is just less than $-1$, and it has been shown in \cite{chevyrev2024} that solutions to (\ref{eq:Allen-Cahn-archetype}) with initial data in $C^{-\alpha}$, $\alpha<-2/3$, may blow up at arbitrarily small times. 

On the other hand, Gabriel, Rosati, and Zygouras in \cite{GRZ24} were able to obtain interesting limits of the system (\ref{eq:Allen-Cahn-archetype}) when $\eta$ is replaced by a scale-$\varepsilon$ mollification $\eta_\varepsilon$ and $\lambda$ is replaced by a logarithmically attenuated coupling constant $\lambda_\varepsilon\coloneqq \lambda/\sqrt{\log \eps^{-1}}$:
\begin{equation}
   \label{eq:Allen-Cahn-weak}
      \begin{cases}
          \partial_t u_\varepsilon=\frac12\Delta u_\varepsilon+\mathfrak{m} u_\varepsilon - \frac{\lambda^2}{\log\eps^{-1}} u_\varepsilon^3          ;\\
          u_\varepsilon(0,x)=\eta_\varepsilon(x)\coloneqq G_{\varepsilon^2}*\eta(x),
      \end{cases}
    \end{equation}
    and then $\eps$ is taken to $0$. Here and henceforth, we use the notation
    \[
      G_t(x) = \frac1{2\pi t}e^{-|x|^2/(2t)}
  \] for the standard heat kernel.
    The mollification and logarithmic attenuation performed in (\ref{eq:Allen-Cahn-weak}) is similar to that performed in recent work on critical SPDE, in particular the stochastic heat equation (and equivalently directed polymers) \cites{MR1629198,MR3719953,MR4666296,MR4602000,dunlap20232dnonlinearstochasticheat,MR4413215,MR4294278,tsai2025stochasticheatflowmoments,clark2025conditionalgmcstochasticheat,gu2025stochasticheatflowblack}, the isotropic and anisotropic KPZ equation \cites{MR4089501,MR4112709,MR4058958,MR4203334,MR4679957}, and the Burgers equation \cite{MR4719971}, all considered in two spatial dimensions. In \cite{GRZ24}, the authors showed that as $\eps\to 0$, the solutions to (\ref{eq:Allen-Cahn-weak}) converge in the mean-square sense to solutions of a heat equation with initial data $\eta$ and a constant potential given by the solution to an explicit ODE. In fact, the authors of \cite{GRZ24} point out that the limiting problem can be cast as a McKean--Vlasov equation \cite[(1.7)]{GRZ24} involving a projection of the nonlinearity onto the first Wiener chaos; see (\ref{eq:McK-V}) below. However, the McKean--Vlasov equation plays little role in the analysis in \cite{GRZ24}, which works via an analysis of the Wild expansion of solutions to (\ref{eq:Allen-Cahn-weak}). Indeed, \cite{GRZ24} poses the question of obtaining ``a deeper understanding of the relations'' between the McKean--Vlasov equation and the original PDE (\ref{eq:Allen-Cahn-weak}). We will provide a more precise description of the results of \cite{GRZ24} after stating our main results below.
    
    In the present paper, we provide a new perspective on the work \cite{GRZ24} by generalizing their results using a completely different proof strategy. Indeed, our proof uses the McKean--Vlasov equation in a central way, replacing the study of Wild expansions in \cite{GRZ24}. Since we do not rely at all on Wild expansions, our major generalization is that we can handle more general (not necessarily polynomial) nonlinearities under certain growth and regularity conditions; see Definition~\ref{def:S} below. We study the similarly mollified and logarithmically-attenuated version of (\ref{eq:archetype}), namely
\begin{equation}
    \label{eq:u-Theta-eqn}
      \begin{cases}
          \partial_t u_\eps=\frac12\Delta u_\eps+\mathfrak{m} u_\eps - \frac{F((t+\eps^2)^{1/2}u_\eps)}{(t+\eps^2)^{3/2}\log\eps^{-1}};\\
          u_\eps(0,x)=\eta_\eps(x).
      \end{cases}
    \end{equation}
In fact, for our first main theorem, we can be even more general, and consider the problem
\begin{equation}
    \label{eq:u-f-eqn}
      \begin{cases}
      \partial_t u_\eps=\frac12\Delta u_\eps+\mathfrak{m} u_\eps - \frac{f(t+\eps^2,u_\eps)}{\log\eps^{-1}};\\
          u_\eps(0,x)=\eta_\eps(x)
      \end{cases}
    \end{equation}
  for a space-time nonlinearity $f=f(t,u)$ satisfying conditions that match the scaling in (\ref{eq:u-Theta-eqn}). Of course, (\ref{eq:u-Theta-eqn}) is recovered from (\ref{eq:u-f-eqn}) by taking
  \begin{equation}\label{eq:f-Theta-self-similar}
    f(t,u) = t^{-3/2}F(t^{1/2}u),
  \end{equation}
and the reader should keep this case in mind throughout even when this particular form of $f$ is not assumed. In the sequel, by $f'$ we will always mean the derivative of $f$ with respect to the second argument, i.e.
\[
  f'(t,u) = \partial_u f(t,u).
\]

The following definition encodes the conditions that we will need to impose on $f$. Because our primary focus is reaction terms of the form (\ref{eq:f-Theta-self-similar}), we will write these conditions in terms of the rescaled function $F$ defined so that
\begin{equation}
\label{eq:f-Theta-general}
	f(t,u) =: t^{-3/2}F(t,t^{1/2}u).
\end{equation}
The self-similar case (\ref{eq:f-Theta-self-similar}) is recovered whenever the function $F(t,u)$ is independent of $t$.
\begin{definition}
  \begin{enumerate}
    \item For $L_1,L_2\geq 0$, we say that $F\in S_1(L_1,L_2)$ if $F$ is continuous, and for each fixed $t>0$, $F(t,\cdot)$ is an odd, differentiable function with Lipschitz derivative such that 
  \begin{equation}
  \label{S1Bd-F}
    \Lip(F(t,\cdot))\le L_1\qquad\text{and}\qquad \Lip(F'(t,\cdot))\le L_2.
  \end{equation}
  We define
  \begin{equation}
      S_1:= \bigcup_{L_1,L_2\in [0,\infty)}S_1(L_1,L_2).
  \end{equation}
  \item For $\gamma_1\ge2$, $\gamma_2\ge 1$, and $\ell_1,l_2\ge 0$, we say that $F\in S_2(\gamma_1,\gamma_2,\ell_1,\ell_2)$ if $F$ is continuous and for each fixed $t>0$, $F(t,\cdot)$ is an odd, differentiable, \emph{increasing} function with locally Lipschitz derivative such that for all $u\ge 0$,
  \begin{equation}
  \label{S2Bd-F}
    \Lip(F(t,\cdot)|_{[0,u]})\le \ell_1(1+|u|^{\gamma_1})\qquad\text{and}\qquad \Lip(F'(t,\cdot)|_{[0,u]})\le \ell_2 (1+|u|^{\gamma_2}).
  \end{equation}
  We define
  \begin{equation}
    S_2 \coloneqq \bigcup_{\substack{\gamma_1\in [2,\infty)\\\gamma_2\in [1,\infty)}}\bigcup_{\substack{\ell_1\in [0,\infty)\\\ell_2\in [0,\infty)}}S_2(\gamma_1,\gamma_2,\ell_1,\ell_2)\qquad\text{and}\qquad
    S_2' \coloneqq \bigcup_{\substack{\gamma_1\in [2,3)\\\gamma_2\in [1,2)}} \bigcup_{\substack{\ell_1\in [0,\infty)\\\ell_2\in [0,\infty)}}S_2(\gamma_1,\gamma_2,\ell_1,\ell_2).
  \end{equation}
  The condition that $\gamma_1<3$ in the definition of $S_2'$ is implied by the condition that $\gamma_2<2$, but we include it for concreteness.
  \item We define $S$ to  be the set of all $F$ such that there are functions $F_1\in S_1$ and $F_2\in S_2$ such that $F = F_1 + F_2$, and similarly $S'$ to be the set of all $F$ such that there are $F_1\in S_1$ and $F_2\in S_2'$ such that $F = F_1 + F_2$. We say that $f$ is an element of $\bar{S}_1$, $\bar{S}_2$,  $\bar{S}_2'$, $\bar{S}$, or $\bar{S}'$ if the corresponding $F$ given by (\ref{eq:f-Theta-general}) is an element of $S_1$, $S_2$,    $S_2'$, $S$ or $S'$, respectively. 
\end{enumerate}
  \label{def:S}
\end{definition}
\begin{remark}
    For a reaction term $f$ and corresponding $F$ given by (\ref{eq:f-Theta-general}), the condition (\ref{S1Bd-F}) is equivalent to the condition that
    \begin{equation}
    \label{S1Bd}
        \Lip(f(t,\cdot))\leq t^{-1}L_1\quad\text{and}\quad \Lip(f'(t,\cdot))\leq t^{-1/2}L_2.
    \end{equation}
    Similarly, (\ref{S2Bd-F}) is equivalent to the condition that
     \begin{equation}
     \label{S2Bd}
    \Lip(f(t,\cdot)|_{[0,u]})\le \ell_1(t^{-1}+|u|^{\gamma_1}t^{\frac{\gamma_1-2}{2}})\quad\text{and}\quad \Lip(f'(t,\cdot))|_{[0,u]})\le \ell_2(t^{-1/2}+|u|^{\gamma_2}t^{\frac{\gamma_2-1}{2}}).
  \end{equation}
  It will also be important that the condition that $F$ is increasing in the second argument means that $f$ is as well, and so in particular if $f\in \bar S_2$, then \begin{equation}f'(t,u)\ge 0.\label{eq:S2derivpos}\end{equation}
\end{remark}

To state our main theorem, we introduce the McKean--Vlasov equation
\begin{equation}
    \label{eq:McK-V}
      \begin{cases}
      \partial_t v_\eps=\frac12\Delta v_\eps+\mathfrak{m} v_\eps - \frac1{\log\eps^{-1}} \mathbb{E}[f'(t+\eps^2,v_\eps)] v_\eps;\\
          v_\eps(0,x)=\eta_\eps(x).
      \end{cases}
    \end{equation}
  Roughly speaking, this equation is obtained from (\ref{eq:u-f-eqn}) by ``projecting the nonlinearity onto the first Wiener chaos assuming that the solution is Gaussian.''
  Of course, the solution to (\ref{eq:u-f-eqn}) is in general \emph{not} Gaussian, but the solution to (\ref{eq:McK-V}) is, and the term $\mathbb{E}[f'(t+\eps^2,v_\eps)] v_\eps = \frac{\mathbb{E}[f(t+\eps^2,v_\eps)v_\eps]}{\mathbb{E}[v_\eps^2]} v_\eps$ (recalling the Gaussian integration by parts formula) is indeed the projection of $f(t+\eps^2,v_\eps)$ onto the first Wiener chaos. We will show in Section~\ref{ExistenceUniquenessSection} below that, assuming that $f\in S$, the problem (\ref{eq:McK-V}) is well-posed, and in fact the solution is given by
    \begin{equation}\label{eq:v-explicit}
      v_\eps(t,x) = \sigma_\eps(t) G^{\mathfrak{m}}_{t}*\eta_\eps(x),
    \end{equation}
    where we have defined\[
      G^{\mathfrak m}_t(x) = e^{\mathfrak m t}G_t(x)
  \] and $\sigma_\eps$ solves (and indeed is the unique solution to) the ODE 
    \begin{equation}\label{eq:sigmaepsODE}
      \begin{cases}
      \dot\sigma_\eps(t) = -\frac1{\log\eps^{-1}}\mathbb{E}[f'(t+\eps^2,\sigma_\eps(t)G^\mathfrak{m}_{t}*\eta_\eps(x))]\sigma_\eps(t);\\
      \sigma_\eps(0)=1.
    \end{cases}
    \end{equation}
In (\ref{eq:sigmaepsODE}) and throughout the paper, we use $*$ to denote spatial (never space-time or temporal) convolution. 

    Our first main result shows that the McKean--Vlasov equation (\ref{eq:McK-V}) approximates well the PDE (\ref{eq:u-f-eqn}):
    \begin{theorem}\label{thm:main-theorem-McK-V}
 Suppose that $f\in \bar{S}'$. Then we have, for any $T_0\in (0,\infty)$, that
  \begin{equation}
  \label{eq:main-theorem-McK-V}
  \adjustlimits\lim_{\eps\to 0}\sup_{T\in [0,T_0]}\sup_{X\in\mathbb{R}^2} 
  (T+\eps^2)^{1/2}|u_\eps(T,X)-v_\eps(T,X)|_{L^2(\mathbb{P})} =0.
  \end{equation}
\end{theorem}

We will see that the typical size of the fluctuations of $u_\eps(T,X)$ and $v_\eps(T,X)$ are on the order of $(T+\eps^2)^{-1/2}$, so the factor of $(T+\eps^2)^{1/2}$ in (\ref{eq:main-theorem-McK-V}) is the appropriate normalization for the error. Also, since $u_\eps(T,X)$ and $v_\eps(T,X)$ are stationary in $X$, the $\sup$ over $X$ in (\ref{eq:main-theorem-McK-V}) is somewhat superfluous, but we include it for completeness.

Theorem~\ref{thm:main-theorem-McK-V} is most useful when we can understand the limiting behavior of $v_\eps$ as $\eps\to 0$. Given~(\ref{eq:v-explicit}), this amounts to studying the behavior of the ODE (\ref{eq:sigmaepsODE}) as $\eps\to 0$. Due to the attenuation by $\log\eps^{-1}$, this is a singular limit, and the ODE does not approach a limiting ODE directly. In our next theorem, we show that it does approach a limit after the change of variables \begin{equation}\label{eq:tchgvar}t = \eps^{2-q}-\eps^2.\end{equation} This change of variables arises due to the scaling properties of the problem in two dimensions and is similar to the natural time scale appearing in many other systems with similar scaling; see for example the discussion in \cite{MR4413215}. This singular change of variables means that we must now impose the self-similar form (\ref{eq:f-Theta-self-similar}) on $f$. Define, for $\eps>0$, the change of variables
\begin{equation}
  \bar{\sigma}_\eps(q) \coloneqq \sigma(\eps^{2-q}-\eps^2).
\end{equation}
\begin{theorem}
\label{thm: LimitingOde}
  Suppose that $F\in S$ and $f$ is given by (\ref{eq:f-Theta-self-similar}). Then we have
  \begin{equation}\label{eq:sigmabareqn}
    \adjustlimits\lim_{\eps\to0}\sup_{q\in [0,2]}|\bar{\sigma}_{\eps}(q)-\bar\sigma(q)|  =0,
  \end{equation}
  where $\bar\sigma$ is the unique solution to the ODE problem
  \begin{equation}
    \begin{cases}
      \dot{\bar\sigma} (q) = - \mathbb{E}[F'( \bar\sigma(q)G_1*\eta(0))]\bar{\sigma}(q);\\
      \bar\sigma(0) =1.
    \end{cases}\label{eq:sigmabarODE}
  \end{equation}
  \end{theorem}

  As a special case, when $F(u) = \lambda^2 u^3$, the ODE in (\ref{eq:sigmabarODE}) becomes
  \[
  \dot{\overline\sigma}(q) = -\frac{3\lambda^2}{4\pi} \overline\sigma(q)^3,
  \] and so
  \begin{equation}
  \overline\sigma(q) = \left(1+\frac{3q\lambda^2}{2\pi}\right)^{-1/2}.\label{eq:overlinesigmaexplicit}
  \end{equation}Thus we recover the limiting behavior of the Allen--Cahn equation studied in \cite{GRZ24}:
\begin{corollary}\label{cor:AC-recover}Let $\lambda \in (0,\infty)$ and
    let $u_\eps$ be the solution to 
\begin{equation}
    \begin{cases}
    \label{eq:allen-cahn-corollary}
        \partial_t u_\eps = \frac{1}{2}\Delta u_\eps + \mathfrak{m}u_\eps - \frac{\lambda^2}{\log\eps^{-1}}u_\eps^3;\\
        u_\eps(0,x) = \eta_\eps(x).
    \end{cases}
\end{equation}
Then for any $T_0\in (0,\infty)$, we have 
\begin{equation}
        \label{eq:allenCahnCorollaryResult-full}
         \lim_{\varepsilon\to 0}\sup_{\substack{T\in(0,T_0]\\X\in \mathbb{R}^2}} (T+\eps^2)^{1/2}\left|u_\varepsilon(T,X) -  \left(1+\frac{3(2-\log_\eps T)}{2\pi}\lambda^2\right)^{-1/2}G^\mathfrak{m}_{T}*\eta(X)\right|_{L^2(\mathbb{P})} = 0.
        \end{equation}
       In particular, for any fixed $T_0>T_0'>0$, we have \begin{equation}
        \label{eq:allenCahnCorollaryResult}
        \lim_{\varepsilon\to 0}\sup_{\substack{T\in(T_0',T_0]\\X\in \mathbb{R}^2}} \left|u_\varepsilon(T,X) -  \left(1+\frac{3}{\pi}\lambda^2\right)^{-1/2}G^\mathfrak{m}_{T}*\eta(X)\right|_{L^2(\mathbb{P})} = 0.
        \end{equation}
\end{corollary}
Equation~(\ref{eq:allen-cahn-corollary}) is equivalent to Equation~(1.3) of \cite{GRZ24} by the rescaling $u_\varepsilon\mapsto \lambda^{-1}u_\varepsilon$ and $F(u)\mapsto \lambda^{-1}F(\lambda u)$. Thus (\ref{eq:allenCahnCorollaryResult}) coincides with the limit appearing in Theorem~1.1 of \cite{GRZ24}. We note that \cite{GRZ24} imposes the additional hypothesis that $\max\{\mathfrak m,0\}T\le \log(\lambda_{\mathrm{fin}}/\lambda)$ for a constant $\lambda_{\mathrm{fin}}\in (0,\infty)$. Our method removes this restriction.

     \subsection{Sketch of the proof}
     The proof of Theorem \ref{thm:main-theorem-McK-V} uses two main ideas. First, the solution $u_\eps$ to (\ref{eq:u-f-eqn}) should in some sense have Gaussian leading order behavior, and second, if we assume the solution is Gaussian, then the nonlinearity $f(t,u_\eps)$ should be close to the corresponding ``projection onto the first chaos," i.e. $\mathbb{E}[f'(t,u_\eps)]u_\eps$. Both of these ideas should be viewed as consequences of the spatial and temporal decorrelation resulting from the initial condition. Indeed, because the nonlinearities can be decomposed into functions that are either uniformly $C^{1,1}$ or contracting, we expect our solutions to become decorrelated at a scale no greater than if we evolved our initial data according to the heat equation. The spatial averaging from the Laplacian of approximately independent random variables then suggests that the leading order behavior of solutions to (\ref{eq:u-f-eqn}) is close to Gaussian and should evolve roughly according to the McKean--Vlasov equation. We are able to capture this behavior at the level of the Malliavin derivative. Indeed,  the signed initial data of the corresponding PDE and the comparison principle allow us to show (in Lemma~\ref{ptwMallbd} below) that
    \begin{equation}
    \label{mallbdIntro}
        0\leq D_zu_\varepsilon(t,x)\leq  e^{3L_1+\mathfrak{m}t}G_{t+\varepsilon^2}(x-z).
    \end{equation}
    
    With this estimate in hand, we begin by proving Theorem~\ref{thm:main-theorem-McK-V} in the case that the nonlinearity belongs only to $\bar{S}_1$ (see Definition \ref{def:S}):
    \begin{proposition}
    \label{lipSoltnBd}
        For each $T_0\in (0,\infty)$ and $\mathfrak{m}\in\mathbb{R}$, there is a constant $C_{T_0,\mathfrak{m}}<\infty$ such that if
        \begin{equation}
        \label{eq:delta-eps-condition}
            \frac{1}{(\log\eps^{-1})^{1/2}}\leq \min\left\{(2e^{|\mathfrak{m}|T_0}L_1)^{-1},\left(\frac{e^{3L_1}L_2}{1-L_1(\log\eps^{-1})^{-1/2}}\right)^{-1}\right\},
        \end{equation}
     then
     the following holds. Suppose that $f\in \bar{S}_1$ with $C^{1,1}$ constants $L_1$ and $L_2$.  Then for any $X\in\mathbb{R}^2$, we have
        \begin{equation}
        \label{eq:lip-soltn-bd}
           \sup_{T\in [0,T_0]}(T+\eps^2)^{1/2} |u_\varepsilon(T,X)-v_\varepsilon(T,X)|_{L^2(\mathbb{P})}\leq  \frac{C_{T_0,\mathfrak{m}}e^{18L_1}L_2}{(\log\varepsilon^{-1})^{1/4}}\exp\left\{\frac{3 e^{3L_1+|\mathfrak{m}|T}L_2}{1-(L_1+e^{3L_1}L_2)(\log\eps^{-1})^{-1/2}}\right\}.
        \end{equation}
    \end{proposition}
     To prove Proposition~\ref{lipSoltnBd},
     we  decompose the accumulated error from the Malliavin derivatives into a piece which captures the leading order Gaussian behavior and terms which depend on the difference of the two solutions and their corresponding Malliavin derivatives. Using decorrelation estimates provided by the Gaussian Poincar\'e covariance inequality, we are able to bound the amount of error in the leading order term and then translate our bounds at the level of the Malliavin derivatives into corresponding bounds on the difference of the two solutions. The error is still difficult to control over the entire interval at once. In fact, our estimate (\ref{mallbdIntro}) implies that the typical size of $u_\eps$ is at most of order $(t+\eps^{2})^{-1/2}$, and in view of our logarithmic attenuation in (\ref{eq:Allen-Cahn-weak}), this means that the bulk of the error is generated at times which are of order $o(1)$, where the nonlinearity has a non-trivial effect on the solution. To account for this fact, we use an exponential timescale similar to that used in \cites{MR3719953,MR4413215} (see Section~\ref{RelTimeScale} below) and construct an iterative scheme which allows us to control the amount of error generated over the short intervals discussed in Section \ref{RelTimeScale}.
     
    We then extend the result by approximation. The contractive properties of the nonlinearities in $\bar{S}_2$ (see Definition \ref{def:S}) mean that the solution to (\ref{eq:u-f-eqn}) will admit a concentration bound that is strong enough for us to add a cutoff to the nonlinearity. The cut-off nonlinearity will satisfy the conditions to be in $\bar{S}_1$, and the quantitative rate of convergence and explicit dependence on $L_1$ and $L_2$ given by Proposition \ref{lipSoltnBd} allow us to show that we still have convergence to the McKean--Vlasov equation even though that the $C^{1,1}$ constants for the cut-off problem must naturally grow as we take $\eps \to 0$.
   \subsection{Relevant timescales}
    \label{RelTimeScale}
    In the proofs of Proposition \ref{lipSoltnBd} and Theorem \ref{thm:main-theorem-McK-V}, it will helpful to view the accumulated error at the final time $T\in (0,T_0]$ in terms of the error accumulated at smaller time scales. Later in the proof, we will see that the decorrelation provided by our estimate (\ref{mallbdIntro}) implies that the typical contribution from the nonlinearity of (\ref{eq:u-f-eqn}) over an interval $[t,t']$ will be of order
    \begin{equation}
    \label{typical-contribution}
     \frac{\log\left(\frac{t'+\eps^2}{t+\eps^2}\right)}{(T+\eps^2)^{1/2}\log\eps^{-1}}
    \end{equation}
    at the final time $T$. This estimate implies that (as $\eps\to 0$) the contribution of the nonlinearity at time scales of order $1$ is negligible: the nonlinearity's effect is happening in an $o(1)$ layer around $t=0$. More precisely, (\ref{typical-contribution}) motivates us to focus on the exponential timescale (\ref{eq:tchgvar}). For both proofs, we will iteratively construct error bounds over intervals that are short enough for us to prove an estimate of order (\ref{typical-contribution}). To this end, we define the sequence of times
	\begin{equation}
    \label{eq:timescale}
		t_m(\varepsilon) = \varepsilon^{2-m\delta_\varepsilon}-\varepsilon^2
	\end{equation}
	for $0\leq m\leq M_\varepsilon$, where 
    \begin{equation}
      \delta_\varepsilon = \frac{1}{\sqrt{\log\varepsilon^{-1}}}\quad\text{and}\quad M_\varepsilon := \left\lfloor \left(2+\frac{\log(T+\varepsilon^2)}{\log\varepsilon^{-1}}\right)\delta_\varepsilon^{-1}\right\rfloor.\label{eq:deltaMepsdefs}
    \end{equation}
    To simplify later calculations, we will always assume that $\varepsilon$ is sufficiently small so that 
    \begin{equation}
        M_\varepsilon \leq 3\delta_{\varepsilon}^{-1},
    \end{equation}
    and to simplify notation, we will not write the dependence on $\eps$ in (\ref{eq:timescale}). The following identity will  help us bound the accumulated error on short intervals:
	\begin{equation}
    \log\left(\frac{t_{m+1}+\varepsilon^2}{t_m+\varepsilon^2}\right) = \delta_\varepsilon\log\varepsilon^{-1}.\label{eq:reasonfortms}
	\end{equation}
\subsection{Organization of the paper}
     In Section~\ref{MckeanVlasovSection}, we establish a well-posedness result for (\ref{eq:McK-V}) and derive an ODE for the limiting behavior of the unique solution. In Section \ref{prelimEstimates}, we derive a pointwise bound on the Malliavin derivative and a corresponding concentration result for our solutions. In Section~\ref{GronwallSection}, we provide Gr\"onwall type estimates which will be used in the proofs of Proposition~\ref{lipSoltnBd} and Theorem~\ref{thm:main-theorem-McK-V}. In Section~\ref{proof-lipschitz}, we prove that solutions to (\ref{eq:u-f-eqn}) converge to the solution of the McKean--Vlasov equation (\ref{eq:McK-V}) in the special case where the nonlinearity $f$ belongs only to $\bar{S}_1$. In Section~\ref{mainResultSection}, we prove Theorem~\ref{thm:main-theorem-McK-V} by approximation using Proposition~\ref{lipSoltnBd}. 
    \subsection{Acknowledgments}
    We warmly thank Simon Gabriel, Tommaso Rosati, and Nikos Zygouras for interesting conversations about their work, as well as Yu Gu for pointers to the literature. A.D.\ was partially supported by the National Science Foundation under grant no.~DMS-2346915.

     \section{The McKean--Vlasov equation}
         \label{MckeanVlasovSection}
     In this section, we will establish the well-posedness of the McKean--Vlasov equation (\ref{eq:McK-V}), and derive a limiting equation in the special case where the non-linearity has the self-similar form (\ref{eq:f-Theta-self-similar}). We will work in the same solution space as the one used in \cite{GRZ24}. That is, we let $\mathcal{E}$ be the class of functions $g$ on $\mathbb{R}^2$ such that there exists some $\lambda = \lambda(g)>0$ with
        \begin{equation}
            \sup_{x\in \mathbb{R}^2}|g(x)|e^{-\lambda |x|}<\infty,
        \end{equation}
     and we define the notion of solution to the McKean--Vlasov equation  (\ref{eq:McK-V}) as in \cite{GRZ24}:
    \begin{definition}
    \label{McKeanVlasovSolDef}
        We say that $v_\varepsilon$ is a solution to (\ref{eq:McK-V}) if it is smooth on $\mathbb{R}_{>0}\times \mathbb{R}^2$, satisfies (\ref{eq:McK-V}), and we have $\sup_{t\in[0,T]}|v_\varepsilon(t,\cdot)|\in \mathcal{E}$ and $\sup_{t\in [0,T]}\mathbb{E}[|f'(t+\eps^2,v_\varepsilon(t,\cdot))|]\in \mathcal{E}$, $\mathbb{P}$-almost surely.
    \end{definition} 
    \subsection{Existence and uniqueness}
    \label{ExistenceUniquenessSection}
      We will begin with a uniqueness result for solutions to the McKean--Vlasov equation (\ref{eq:McK-V}) under the assumption that $f\in \bar{S}$. (See Definition~\ref{def:S}.)
         \begin{lemma}
        \label{MckeanUnique}
      Suppose $f\in \bar{S}$, and in particular that $f=f_1+f_2$ for some $f_1\in \bar{S}_1(L_1,L_2)$ and $f_2\in \bar{S}_2(\gamma_1,\gamma_2,\ell_1,\ell_2)$. Then there is at most one solution to the McKean-Vlasov equation 
            \begin{equation}
            \label{mckeanvlasovUniqueEq}
		\begin{cases}
			\partial_t v_\varepsilon = \frac12\Delta v_\varepsilon + \mathfrak{m} v_\varepsilon +  \frac{1}{\log\varepsilon^{-1}}\mathbb{E}[f'(t+\eps^2,v_\varepsilon)]v_\varepsilon;\\
			v_\varepsilon(0,\cdot) = \eta_\eps(x).
		\end{cases}
	\end{equation}
        \end{lemma}
		\begin{proof}
        We first note that essentially the same argument as the one used in \cite[Proposition~6.2]{GRZ24} can be used to show that for any solution $v_\eps$ to (\ref{mckeanvlasovUniqueEq}) and $p\in [1,\infty)$, there exists a    $K_{p,T}(\varepsilon)$ depending only on $\varepsilon$, $T$,$\mathfrak{m}$, $L_1$, and $p$ such that 
         \begin{equation}
        \label{eq:a-priori-momentbd}
\adjustlimits\sup_{t\in [0,T]}\sup_{x\in \mathbb{R}^2}|v_\varepsilon(t,x)|_{L^p(\mathbb{P})}\leq K_{p,T}(\varepsilon).
        \end{equation}
        This $\eps$-dependent bound is of course highly suboptimal, and we will upgrade this to a stronger $\eps$-independent result in Lemma~\ref{concentrationIneq} below.
	   Now suppose $v_\varepsilon$ and $v_\varepsilon^*$ both solve (\ref{mckeanvlasovUniqueEq}). We write the difference $\omega_\varepsilon := v_\varepsilon - v^*_\varepsilon$, which solves the equation
     \begin{align}
       \begin{split}
               \partial_t \omega_\eps &= \frac{1}{2}\Delta \omega_\eps + \mathfrak{m} \omega_\eps - \frac{1}{\log\varepsilon^{-1}}\mathbb{E}[f'(t+\eps^2,v_\varepsilon)]v_\varepsilon + \frac{1}{\log\varepsilon^{-1}}\mathbb{E}[f'(t+\eps^2,v_\varepsilon^*)]v_\varepsilon^*\\
               &= \frac{1}{2}\Delta \omega_\eps + \mathfrak{m} \omega_\eps - \frac{1}{\log\varepsilon^{-1}}\mathbb{E}[f'(t+\eps^2,v_\varepsilon)]\omega_\varepsilon  - \frac{1}{\log\varepsilon^{-1}}\mathbb{E}[f'(t+\eps^2,v_\varepsilon)-f'(t+\eps^2,v_\varepsilon^*)]v_\varepsilon^*
     \end{split}
       \label{expandedPDEforV}
     \end{align}
       with the initial condition
       \begin{equation}
           \omega_\eps(0,\cdot) = 0.
       \end{equation}
       We will use the method of sub-/super-solutions to bound the error, which will allow us to eliminate the  contracting term $-\frac{1}{\log\varepsilon^{-1}}\mathbb{E}[f_2'(t+\eps^2,v_\varepsilon)]\omega_\varepsilon$ from (\ref{expandedPDEforV}). 
       Using the assumption that $f_1\in \bar{S}_1(L_1,L_2)$ and $f_2\in \bar{S}_2(\gamma_1,\gamma_2,\ell_1,\ell_2)$, we have
       \begin{equation}
           \begin{split}
               |\mathbb{E}[f'(t+\eps^2,v_\varepsilon)-f'(t+\eps^2,v_\varepsilon^*)]| &\leq  |\mathbb{E}[f_1'(t+\eps^2,v_\varepsilon)-f_1'(t+\eps^2,v_\varepsilon^*)]|
               \\&\qquad + |\mathbb{E}[f_2'(t+\eps^2,v_\varepsilon)-f_2'(t+\eps^2,v_\varepsilon^*)]|\\
               &\leq \frac{L_2+\ell_2}{\sqrt{t+\varepsilon^2}}|\omega_\varepsilon|_{L^2(\mathbb{P})} + 2\ell_2(t+\varepsilon^2)^{\frac{\gamma_2-1}{2}}\mathbb{E}\left[|\omega_\varepsilon|(|v_\varepsilon|^{\gamma_2}+|v^*_\varepsilon|^{\gamma_2})\right].
           \end{split}
       \end{equation}
       Then by our moment bound for both $v_\varepsilon$ and $v^*_\varepsilon$ from (\ref{eq:a-priori-momentbd}) and  the Cauchy--Schwarz inequality, this implies
       \begin{equation}
           \begin{split}
               |\mathbb{E}[f'(t+\eps^2,v_\varepsilon)-f'(t+\eps^2,v_\varepsilon^*)]| 
               &\leq C(\eps)|\omega_\varepsilon|_{L^2(\mathbb{P})},
           \end{split}
       \end{equation}
       where $C(\eps)$ is a constant depending on $\eps$ as well as $\ell_2$, $L_2$, $\gamma_2$, and $T$. 
        Now, we define the sub-solution $\underline{\omega}_\varepsilon$ as the solution of 
        \begin{equation}
        \label{v*subSoltnEq}
           \begin{cases}
             \partial_t  \underline{\omega}_\varepsilon = \frac{1}{2}\Delta \underline{\omega}_\varepsilon  + \left(\mathfrak{m}+\frac{L_1}{(t+\eps^2)\log\eps^{-1}}\right)\underline{\omega}_{\varepsilon}    - \frac{C(\eps)}{\log\varepsilon^{-1}}|\omega_\varepsilon|_{L^2(\mathbb{P})}|v^*_\varepsilon|\\
    \underline{\omega}_\varepsilon(0,\cdot) = 0.
           \end{cases}
       \end{equation}
       The problem (\ref{v*subSoltnEq}) is linear and so has an easy solution theory via the Duhamel formula. The maximum principle (see e.g.\ \cite[Theorem~9 of Chapter~2]{MR181836}) implies that $\underline{\omega}_\eps\leq 0$. Also, if we consider the parabolic operator
    \begin{equation}
      L\omega := \frac{1}{2}\Delta \omega + \mathfrak{m}\omega -\frac{1}{\log\eps^{-1}}\mathbb{E}[f'(t+\eps^2,v_\eps)]\omega - \partial_t\omega,
    \end{equation}
    then we see that
    \begin{equation}
        \begin{split}
            L[\omega_\eps - \underline{\omega}_\eps] &= \frac{1}{\log\eps^{-1}}\mathbb{E}[f'(t+\eps^2,v_\eps)]\underline{\omega}_\eps + \frac{L_1}{(t+\eps^2)\log\eps^{-1}}\underline{\omega}
            \\&\qquad+ \frac{1}{\log\eps^{-1}}\left(\mathbb{E}[f'(t+\eps^2,v_\varepsilon)-f'(t+\eps^2,v_\varepsilon^*)]v_\varepsilon^* - C(\eps)|\omega_\eps|_{L^2(\mathbb{P})}|v_\eps^*|\right)\\
            &\leq 0,
        \end{split}
    \end{equation}
    so another application of the maximum principle gives
    \begin{equation}
        \underline{\omega}_\eps(t,x) \leq \omega_\eps(t,x)
    \end{equation}
    for all $t\in [0,T]$ and $x\in \mathbb{R}^2$. Similarly, we define the super-solution $\overline{\omega}_\varepsilon$, which  solves the equation
       \begin{equation}
       \label{uniquenessSupEq}
           \begin{cases}
               \partial_t  \overline{\omega}_\varepsilon = \frac{1}{2}\Delta \overline{\omega}_\varepsilon  + \left(\mathfrak{m} + \frac{L_1}{(t+\eps^2)\log\eps^{-1}}\right) \overline{\omega}_\varepsilon  + \frac{C(\eps)}{\log\varepsilon^{-1}}|\omega_\varepsilon|_{L^2(\mathbb{P})}|v^*_\varepsilon|\\
               \overline{\omega}^*_\varepsilon(0,\cdot) = 0
           \end{cases}
       \end{equation}
       and satisfies
       \begin{equation}
           \omega_\eps(t,x)\leq \overline{\omega}_\eps(t,x)
       \end{equation}
       for all $t\in [0,T]$ and $x\in \mathbb{R}^2$. Then using the fact that
     \begin{equation}
         \frac{1}{\log\varepsilon^{-1}}\int_0^t \frac{L_1}{s+\varepsilon^2}\,ds \leq \frac{L_1}{\log\varepsilon^{-1}}\log\left(\frac{T+\varepsilon^2}{\varepsilon^2}\right)\leq 3L_1,
     \end{equation}
     we apply Duhamel's formula to (\ref{uniquenessSupEq}) and take $|\cdot|_{L^2(\mathbb{P})}$ of both sides, yielding
     \begin{equation}
|\overline{\omega}_\varepsilon(t,x)|_{L^2(\mathbb{P})}\leq \frac{e^{3L_1}C(\eps)}{\log\varepsilon^{-1}}\int_{0}^t \frac{|\omega_\varepsilon(s,x)|_{L^2(\mathbb{P})}|v_\varepsilon^*(s,x)|_{L^2(\mathbb{P})}}{\sqrt{s+\varepsilon^2}}\,ds.
     \end{equation}
     By (\ref{eq:a-priori-momentbd}), this implies that
      \begin{equation}
         |\overline{\omega}_\varepsilon(t,x)|_{L^2(\mathbb{P})}\leq  \tilde{C}(\eps)\int_{0}^t |\omega_\varepsilon(s,x)|_{L^2(\mathbb{P})}\,ds
     \end{equation}
     for some constant $\tilde{C}(\eps)$ depending only on $\eps$, $L_1$, $\ell_1$, $\gamma_2$, $T$, $\mathfrak{m}$, and $L_1$. 
     By symmetry and the fact that $\underline{\omega}_\eps \leq \omega_\eps\leq \overline{\omega}_\eps$, this gives
     \begin{equation}
         |\omega_\varepsilon(t,x)|_{L^2(\mathbb{P})}\leq   \tilde{C}(\eps)\int_{0}^t |\omega_\varepsilon(s,x)|_{L^2(\mathbb{P})}\,ds.
     \end{equation}
   Applying Gr\"onwall's inequality gives
\begin{equation}
     |v_\eps(t,x)-v_\eps^*(t,x)|_{L^2(\mathbb{P})} = 0,
\end{equation}
 which completes the proof of Lemma \ref{MckeanUnique}.
	\end{proof}
 We also provide the following existence result for (\ref{eq:McK-V}). The proof is essentially the same as the proof of Proposition 1.2 in \cite{GRZ24}, which showed existence in the special case of the Allen--Cahn equation.
  \begin{lemma}
    \label{McKeanVlasovExistence}
       Suppose the nonlinearity $f\in S$. Then the solution to (\ref{eq:McK-V}) is
        \begin{equation}\label{eq:vformula}
v_\varepsilon(t,x)=\sigma_\eps(t)G^{\mathfrak{m}}_{t}*\eta_\eps(x),
        \end{equation}
        where $\sigma_\eps$ is the unique solution to the initial value problem 
        \begin{equation}
        \label{epsilonIVP}
        \begin{cases}
            \dot{\sigma}_\eps(t) = -\frac{1}{\log\varepsilon^{-1}}\mathbb{E}[f'(t+\eps^2,\sigma_\eps(t)G^{\mathfrak{m}}_{t}*\eta_\eps(x))]\sigma_\eps(t)\\
            \sigma_\eps(0) = 1.
        \end{cases}
\end{equation}
    \end{lemma}
\begin{proof}

Let us first show that if (\ref{epsilonIVP}) is well-posed, then
\begin{equation}
  v_\varepsilon(t,x) \coloneqq  \sigma_\eps(t)G^{\mathfrak{m}}_{t}*\eta_\eps(x)\label{eq:vepsexplicit}
\end{equation}
satisfies the Duhamel formula corresponding to (\ref{eq:McK-V}). Writing the integral form of (\ref{epsilonIVP}), we have
\begin{equation}
    \sigma_\eps(t) = 1 - \frac{1}{\log\varepsilon^{-1}}\int_0^t \mathbb{E}[f'(s+\eps^2, \sigma_\eps(s)G^{\mathfrak{m}}_{s}*\eta_\eps(x))]\sigma_\eps(s)\,ds.
\end{equation}
Multiplying both sides by $G^{\mathfrak{m}}_{t}*\eta_\eps(x)$ gives
\begin{equation*}
\begin{split}
  \sigma_\eps(t)&G^{\mathfrak{m}}_{t}*\eta_\eps(x) =   G_{t}^\mathfrak{m}*\eta_\eps(x)-\frac{G^{\mathfrak{m}}_{t}*\eta_\eps(x)}{\log\varepsilon^{-1}} \int_0^t \mathbb{E}[f'(s+\eps^2, \sigma_\eps(s)G^{\mathfrak{m}}_{s}*\eta_\eps(x))]\sigma_\eps(s)\,ds\\
    &=  G_{t}^\mathfrak{m}*\eta_\eps(x)-\frac{1}{\log\varepsilon^{-1}}\int_0^t \int G_{t-s}^\mathfrak{m}(x-y)\mathbb{E}[f'(s+\eps^2, \sigma_\eps(s)G^{\mathfrak{m}}_{s}*\eta(y))]
    \sigma_\eps(s)G^{\mathfrak{m}}_{s}*\eta_\eps(y)\,dy\,ds,
\end{split}
\end{equation*}
where we used spatial homogeneity and the Chapman--Kolmogorov relation in the second equality. Recalling (\ref{eq:vepsexplicit}), we can rewrite this as
\begin{equation}
\label{duhamelMckean}
        v_\varepsilon(t,x) = G_t^\mathfrak{m}*\eta_\eps -\frac{1}{\log\varepsilon^{-1}}\int_0^t \int G_{t-s}^\mathfrak{m}(x-y)\mathbb{E}[f'(s+\eps^2,v_\varepsilon(s,y))]v_\varepsilon(s,y)\,dy\,ds.
\end{equation}
We recognize (\ref{duhamelMckean}) as the Duhamel formulation of (\ref{eq:McK-V}).

We now show that the initial value problem (\ref{epsilonIVP}) is well-posed. The local existence for (\ref{epsilonIVP}) follows from the fact that $-\frac{1}{\log\eps^{-1}}\mathbb{E}[f'(t+\eps^2,\sigma G_t^\mathfrak{m}*\eta_\eps(x))]\sigma$ is locally Lipschitz in $\sigma$ by our assumption that $f_1$ satisfies (\ref{S1Bd}) and $f_2$ satisfies (\ref{S2Bd}). To complete the proof, we will show that $\sigma_\eps$ remains bounded over the entire interval $[0,T]$. We first observe that
\begin{equation}
 \begin{split}
         -\frac{1}{\log\varepsilon^{-1}}\mathbb{E}[f'(t+\eps^2,\sigma_\eps(t)G^{\mathfrak{m}}_{t}*\eta_\eps(x))] &= - \frac{1}{\log\varepsilon^{-1}}\sum_{j=1}^2\mathbb{E}[f_j'(t+\eps^2,\sigma_\eps(t)G^{\mathfrak{m}}_{t}*\eta_\eps(x))]\\
     &\leq \frac{L_1(\log\varepsilon^{-1})^{-1}}{t+\varepsilon^2}.
 \end{split}
\end{equation}
As long as $\sigma_\eps$ exists, it must be positive by the comparison principle, so this implies that
\begin{equation}
    \dot{\sigma}_\eps(t)\leq \frac{L_1(\log\varepsilon^{-1})^{-1}}{t+\varepsilon^2}\sigma_\eps(t).
\end{equation}
Hence, the solution $\sigma_\eps(t)$ of (\ref{epsilonIVP}) exists for all $t\in [0,T]$, and we have the uniform in $\varepsilon$ bound
\begin{equation}
\label{alphaEpsilonBound}
    0<\sigma_\eps(t)\leq \exp\left\{\frac{L_1}{\log\varepsilon^{-1}}\int_0^t \frac{1}{s+\varepsilon^2}\,ds\right\}\leq e^{3L_1}.
\end{equation}

The fact that $v_\varepsilon$ solves (\ref{eq:McK-V}) in the sense of Definition \ref{McKeanVlasovSolDef} follows from the proof of Proposition~1.2 in \cite{GRZ24} This completes the proof. 
\end{proof}
We conclude this section with the following remarks. The first will be used in the proof of Proposition~\ref{lipSoltnBd}, and the second shows that the McKean--Vlasov equation (\ref{eq:McK-V}) is a special case of (\ref{eq:u-f-eqn}).
    \begin{remark}
    \label{detMallDer}
        The Malliavin derivative $D_zv_\varepsilon(t,x)$ is deterministic for any $t\in \mathbb{R}_{\geq 0}$. In particular, we have
        \begin{equation}
            D_zv_\varepsilon =  \sigma_\eps(t)e^{\mathfrak{m}t}G_{t+\varepsilon^2}(x-z).
        \end{equation}
    \end{remark}
    \begin{remark}
\label{uvSame}
    Differentiating (\ref{eq:vformula}), we see that  when $f = f_1 + f_2$ with $f_1\in \bar{S}_1(L_1,L_2)$ and $f_2\in \bar{S}_2$, the McKean--Vlasov solution $v_\varepsilon$ solves 
\begin{equation}
    \begin{cases}
        \partial_t v_\eps = \frac12\Delta v_\eps + \mathfrak m v_\eps + \frac{\dot\sigma_\eps(t)}{\sigma_\eps(t)}v_\eps;\\
        v_\eps(0,x) =\eta_\eps(x).
    \end{cases}\label{eq:McKVasoriginal}
\end{equation}
The problem (\ref{eq:McKVasoriginal}) can be considered as a version of (\ref{eq:u-f-eqn}) with $f$ replaced by
\begin{align}
h_\eps(t,v) \coloneqq \frac{\dot\sigma_\eps(t-\eps^2)}{\sigma_\eps(t-\eps^2)}v &= -\frac1{\log\eps^{-1}}\mathbb{E}[f'(t,\sigma_\eps(t-\eps^2)G^{\mathfrak m}_{t-\eps^2}*\eta_\eps(x))]\notag\\
&= -\sum_{j=1}^2\frac1{\log\eps^{-1}}\mathbb{E}[f'_j(t,\sigma_\eps(t-\eps^2)G^{\mathfrak m}_{t-\eps^2}*\eta_\eps(x))],\label{eq:u-v-same}
\end{align}
where we used (\ref{epsilonIVP}). Using (\ref{alphaEpsilonBound}) and our assumptions on $f$, one can easily verify that the first term of the summation in (\ref{eq:u-v-same}) is an element of $\bar{S}_1(L_1,0)$ and the second term in the summation is an element of $\bar{S}_2$. Hence, the McKean--Vlasov equation (\ref{eq:McK-V}) is  a special case of (\ref{eq:u-f-eqn}) with a reaction term in $\bar{S}$. There is a minor subtlety in that $h_\eps$ depends (slightly) upon $\eps$, but this is not a serious issue since the bounds from Definition~\ref{def:S} on $h_\eps$ are independent of $\eps$.
\end{remark}
       \subsection{Limiting behavior: proof of Theorem~\ref{thm: LimitingOde}}
       Now that we have established the existence and uniqueness of solutions to (\ref{eq:McK-V}), we will derive a limiting problem under the stronger assumption that the nonlinearity $f$ has the self-similar form (\ref{eq:f-Theta-self-similar}). By Lemma~\ref{McKeanVlasovExistence}, it will be sufficient to consider the limiting problem for (\ref{epsilonIVP}). The requirement that $f$ be self-similar is a natural result of our logarithmic attenuation in (\ref{eq:McK-V}) which causes the initial value problem (\ref{epsilonIVP}) to have a singular limit. In fact, one might expect that the limiting problem is only well-posed when the equation (\ref{eq:McK-V}) is scale invariant, in which case we would need to take $\mathfrak{m}=0$. However, we will see in the proof that the effect of this term is weak enough on small timescales that the limit can be still be taken for any $\mathfrak{m}\in \mathbb{R}$. 
       
       We will now proceed with the proof of Theorem~\ref{thm: LimitingOde}. Let $F = F_1+F_2\in  S$,  with $F_1\in  S_1(L_1,L_2)$ and $F_2\in S_2(\gamma_1,\gamma_2,\ell_1,\ell_2)$, and let $f$ be given by (\ref{eq:f-Theta-self-similar}). We also define
\begin{equation}
    \bar{\sigma}_\eps(q) := \sigma_\eps(\varepsilon^{2-q}-\varepsilon^2)\quad\text{for}\quad q\in \left[0,2 + \frac{\log(T+\varepsilon^2)}{\log\varepsilon^{-1}}\right],
\end{equation}
where $\sigma_\eps$ is given by (\ref{epsilonIVP}). We see that
\begin{equation}
\label{rescaledODE}
    \begin{split}
        \dot{\bar{\sigma}}_\eps(q) &=  -\varepsilon^{2-q}\mathbb{E}[f'(\eps^{2-q},\bar{\sigma}_\eps(q)e^{\mathfrak{m}(\eps^{2-q}-\eps^2)}G_{\varepsilon^{2-q}}*\eta(x))]\bar{\sigma}_\eps(q)\\
        &= -\varepsilon^{2-q}\int_{-\infty}^\infty \sqrt{2\varepsilon^{2-q}}f'(\eps^{2-q}, \bar{\sigma}_\eps(q)e^{\mathfrak{m}(\eps^{2-q}-\eps^2)}w)\bar{\sigma}_\eps(q)\exp\left\{-2\pi w^2\varepsilon^{2-q}\right\}\,dw.
        \end{split}
        \end{equation}
        By (\ref{eq:f-Theta-self-similar}), this implies 
        \begin{equation*}
        \begin{split}
          \dot{\bar{\sigma}}_\eps(q)  &=  -\varepsilon^{2-q}\int_{-\infty}^\infty \sqrt{2\varepsilon^{2-q}}\frac{F'( \bar{\sigma}_\eps(q)e^{\mathfrak{m} (\eps^{2-q}-\eps^2)}w\sqrt{\varepsilon^{2-q}})}{\varepsilon^{2-q}}\bar{\sigma}_\eps(q)\exp\left\{-2\pi w^2\varepsilon^{2-q}\right\}\,dw\\        &= -\sqrt{2}\int_{-\infty}^\infty F'( \bar{\sigma}_\eps(q)e^{\mathfrak{m}(\eps^{2-q}-\eps^2)}w)\bar{\sigma}_\eps(q)\exp\left\{-2\pi w^2\right\}\,dw\\
          &= -\mathbb{E}[F'( \bar{\sigma}_\eps(q)e^{\mathfrak{m}(\eps^{2-q}-\eps^2)}G_{1}*\eta(x))]\bar{\sigma}_\eps(q),
    \end{split}
\end{equation*}
so $\bar{\sigma}_\eps$ solves the initial value problem
\begin{equation}\label{eq:sigmabarepseqn}
    \begin{cases}
        \dot{\bar{\sigma}}_\eps(q) = -\mathbb{E}[F'( \bar{\sigma}_\eps(q)e^{\mathfrak{m} (\eps^{2-q}-\eps^2)}G_{1}*\eta(x))]\bar{\sigma}_\eps(q)\\
        \bar{\sigma}_\eps(0) = 1.
    \end{cases}
\end{equation}

The case $\mathfrak m=0$ is particularly simple because in this case the problem (\ref{eq:sigmabarepseqn}) does not depend on $\eps$, and in fact already agrees with (\ref{eq:sigmabareqn}), so in this case we do not even need to take the limit $\eps\to 0$. On the other hand, for general $\mathfrak{m}\in\mathbb{R}$, $q\in (0,2)$, and $\eps$ very small, the term $e^{\mathfrak m(\eps^{2-q}-\eps^2)}$ appearing in (\ref{eq:sigmabarepseqn}) is very close to $1$. Thus, it should disappear in the limit $\eps\to 0$. Indeed, we will show that $\bar{\sigma}_\eps \to \bar\sigma$ uniformly on $[0,2]$, where $\bar\sigma$ is the unique solution to 
\begin{equation}
    \begin{cases}
    \label{limitingODE}
           \dot{\bar\sigma}(q) = -\mathbb{E}[F'( \bar\sigma(q)G_{1}*\eta(x))]\bar\sigma(q)\\
         \bar\sigma(0) = 1.
    \end{cases}
\end{equation}

The fact that $(\ref{limitingODE})$ is well-posed follows from an argument similar to the one in the proof of Lemma~\ref{McKeanVlasovExistence}. We now observe that
\begin{equation}
    \begin{split}
        |\bar{\sigma}_\eps(q)-\bar\sigma(q)|&\leq  \int_0^q \Big{|}  \mathbb{E}[F'( \bar{\sigma}_\eps(r)e^{\mathfrak{m} (\eps^{2-r}-\eps^2)}G_{1}*\eta(x))]\bar{\sigma}_\eps(r)
        - \mathbb{E}[F'( \bar\sigma(r)G_{1}*\eta(x))]\bar\sigma(r)\Big{|}\,dr\\
        &\leq \int_0^q \left|  \mathbb{E}[F'( \bar{\sigma}_\eps(r)e^{\mathfrak{m} (\eps^{2-r}-\eps^2)}G_{1}*\eta(x))]\right||\bar{\sigma}_\eps(r)-\bar\sigma(r)|\,dr\\
        &\qquad+\int_0^q \left|  \mathbb{E}[F'( \bar{\sigma}_\eps(r)e^{\mathfrak{m}(\eps^{2-r}-\eps^2)}G_{1}*\eta(x))-F'( \bar\sigma(r)G_{1}*\eta(x))]\right||\bar\sigma(r)|\,dr.
    \end{split}
    \label{eq:sigmaerrbd}
\end{equation}
For the first term, we note that because $0< \sigma_\eps\leq e^{3L_1}$, we have
\begin{equation}
\label{limitingODEBd1}
    \begin{split}
         &\left|  \mathbb{E}[F'( \bar{\sigma}_\eps(r)e^{\mathfrak{m} \varepsilon^{2-r}}G_{1}*\eta(x))]\right| =  \sum_{j=1}^2\left|  \mathbb{E}[F_j'( \bar{\sigma}_\eps(r)e^{\mathfrak{m} (\eps^{2-r}-\eps^2)}G_{1}*\eta(x))]\right|\\
         &\qquad\leq L_1+\ell_1 + \bar{\sigma}_\eps(r)^{\gamma_1}e^{\gamma_1|\mathfrak{m}|\varepsilon^{2-r}}\mathbb{E}[(G_1*\eta(x))^{\gamma_1}]
         \leq C_{1},
    \end{split}
\end{equation}
where $C_{1}$ is a constant that depends only on $\mathfrak{m}$, $L_1$, $\ell_1$, $\gamma_2$, and $T$.
For the second term on the right side of (\ref{eq:sigmaerrbd}), we first observe that
\begin{equation}
\begin{aligned}
    |\bar{\sigma}_\eps(r)e^{\mathfrak{m}\varepsilon^{2-r}}G_1*\eta(x)&- \bar\sigma(r)G_1*\eta(x)|\\&\leq  \left(e^{\mathfrak{m}(\varepsilon^{2-r}-\eps^2)}|\bar{\sigma}_\eps(r)-\bar\sigma(r)| +|\bar\sigma(r)||1-e^{\mathfrak{m}(\varepsilon^{2-r}-\eps^2)}|\right)G_1*\eta(x)\\
   & \leq   (G_1*\eta(x)) \left(e^{\mathfrak{m}(\varepsilon^{2-r}-\eps^2)}|\bar{\sigma}_\eps(r)-\bar\sigma(r)| +|\bar\sigma(r)|(e^{|\mathfrak{m}|\varepsilon^{2-q}}-1)\right).
    \end{aligned}
\end{equation}
Next, we see that
\begin{equation}
  \begin{split}
 \label{limitingODEBd2}
 \sum_{j=1}^2 &|  \mathbb{E}[F_j'( \bar{\sigma}_\eps(r)e^{\mathfrak{m} \varepsilon^{2-r}}G_{1}*\eta(x))-F_j'( \bar\sigma(r)G_{1}*\eta(x))]| \\&\leq  \left(e^{\mathfrak{m}\varepsilon^{2-r}}|\bar{\sigma}_\eps(r)-\bar\sigma(r)| +e^{3L_1}(e^{|\mathfrak{m}|\varepsilon^{2-q}}-1)\right)\\&\hspace{4em}\cdot(L_2+\ell_2)\left(\mathbb{E}[| (G_1*\eta(x)) |] + 2e^{\gamma_2(3L_1+|\mathfrak{m}|T)}\mathbb{E}[(G_1*\eta(x))^{1+\gamma_2}]
      \right)\\
      &\leq C_2\left(e^{\mathfrak{m}\varepsilon^{2-r}}|\bar{\sigma}_\eps(r)-\bar\sigma(r)| +e^{3L_1}(e^{|\mathfrak{m}|\varepsilon^{2-q}}-1)\right),
 \end{split}
           \end{equation}
where $C_{2}$ is a constant that depends only on $\mathfrak{m}$, $L_1$, $L_2$, $\ell_2$, $\gamma_2$, and $T$. Then combining (\ref{limitingODEBd1}) and (\ref{limitingODEBd2}), we have
\begin{equation}
    \begin{split}
        |\bar{\sigma}_\eps(q)-\bar\sigma(q)|\leq C_3e^{3L_1}(e^{|\mathfrak{m}|\varepsilon^{2-q}}-1)
        +
        C_3\int_0^q |\bar{\sigma}_\eps(r)-\bar\sigma(r)|\,dr,
    \end{split}
\end{equation}
where $C_3 = C_1 +  e^{|\mathfrak{m}|}L_2C_2$. Applying Gr\"onwall's inequality gives
\begin{equation}
    |\bar{\sigma}_\eps(q)-\bar\sigma(q)|\leq  C_3e^{3L_1}(e^{|\mathfrak{m}|\varepsilon^{2-q}}-1)e^{C_3 q},
\end{equation}
Taking $\varepsilon$ to zero gives $\lim_{\varepsilon\to 0}\bar{\sigma}_\eps(q) = \bar\sigma(q)$ for any $0\leq q<2$. Because $\bar\sigma$ is continuous, we can conclude that $\lim_{\varepsilon\to 0} \bar{\sigma}_\eps(2) = \bar\sigma(2)$ if we show that $\{\bar{\sigma}_\eps\}_\varepsilon$ is uniformly Lipschitz. This follows from the fact that 
\begin{equation}
\label{eq:limit-eq-unif-lip}
    \begin{split}
        |\dot{\bar{\sigma}}_\eps(q)|&\leq e^{3L_1}\sum_{j=1}^2 |\mathbb{E}[F_j'( \bar{\sigma}_\eps(q)e^{\mathfrak{m} (\varepsilon^{2-q}-\eps^2)} G_1*\eta(x))]|\\
        &\leq e^{3L_1}(L_1+\ell_1) +e^{3L_1+\gamma_1|\mathfrak{m}|T}|\bar{\sigma}_\eps(q)|^{\gamma_1}\mathbb{E}[(G_1*\eta(x))^{\gamma_1}]\\
        &\leq C_{4},
    \end{split}
\end{equation}
where $C_4$ depends only on $\mathfrak{m}$, $\ell_1,L_1,$ and $T$. This implies that $\bar \sigma_\eps $ converges to $\bar \sigma$ uniformly on $[0,2]$, which completes the proof of Theorem \ref{thm: LimitingOde}.

\section{The Malliavin derivative}
    \label{prelimEstimates}
 An important tool in our proof will be the Malliavin calculus. It is convenient to consider the Malliavin differentiated version of (\ref{eq:u-f-eqn}) because, although the initial data for (\ref{eq:u-f-eqn}) is a Gaussian process and thus does not have a sign, the Malliavin derivative is strictly positive. This allows for the use of techniques for parabolic equations like the maximum principle that will help us derive decorrelation estimates for our solutions.

 Let us recall some basic Malliavin calculus definitions from \cite{MR2200233}. We refer to that work for more background.
         We define the spatial white noise $W$ to be Gaussian process indexed by $L^2(\mathbb{R}^2)$ defined so that for any $h_1,h_2\in L^2(\mathbb{R}^2)$, we have
    \begin{equation}
        \mathbb{E}[W(h_1)W(h_2)] = \langle h_1,h_2\rangle_{L^2(\mathbb{R}^2)}.
    \end{equation}
    We will often use the notation $\int h(x) \eta(x)\,dx \coloneqq W(h)$, so that
    \begin{equation}
        \mathbb{E}\left[\left(\int h_1(x)\eta(x)\,dx\right)\left(\int h_2(y)\eta(y)\,dy\right)\right] = \int h_1(x)h_2(x)\,dx.
    \end{equation}
    We define a set $\mathcal{S}$ of ``smooth'' random variables $G$ that take the form 
        \begin{equation}
            G = g(W(h_1),...,W(h_n)),
        \end{equation}
        for some $h_1,...,h_n\in L^2(\mathbb{R}^2)$ and some function $g\in C^\infty(\mathbb{R}^2)$ such that $g$ and all of its partial derivatives have at most polynomial growth at infinity.
        The Malliavin derivative $DG$ of a smooth random variable $G\in \mathcal{S}$ is the $L^2(\mathbb{R}^2)$-valued random variable 
        \begin{equation}
            DG = \sum_{j=1}^n \partial_jg(W(h_1),...,W(h_n))h_j.
        \end{equation}
        For $p\in[1,\infty)$, we let $\mathbb{D}^{1,p}$ denote  the closure of $\mathcal{S}$ with respect to the norm
        \begin{equation}
            \|G\|_{1,p} \coloneqq \left(\mathbb{E}[|G|^p]+\mathbb{E}[|DG|_{L^2(\mathbb{R}^2)}^p]\right)^{1/p},
        \end{equation}
       and we extend the definition of $DG$ to all $G\in\mathbb{D}^{1,p}$ (for any $p\in [1,\infty)$) by density.
    We will also let $\{D_zG\}_{z\in \mathbb{R}^2}$ be the random field associated to $DF$, which is given by the natural identification between $L^2(\Omega;L^2(\mathbb{R}^2))$ and $L^2(\mathbb{R}^2\times \Omega)$.

    In this section, we will derive an \emph{a priori} estimate on the Malliavin derivative of solutions to (\ref{eq:u-f-eqn}) which will be used throughout the paper. This estimate will allow us to use the decorrelation provided by our initial condition. Because we derive a deterministic bound, it also implies a concentration result and a corresponding moment bound for our solutions.
    \begin{lemma}
        \label{ptwMallbd}
        Suppose the nonlinearity $f\in\bar S$ can be written as $f=f_1+f_2$ with $f_1\in \bar{S}_1(L_1,L_2)$ and $f_2\in \bar{S}_2$. Then for all $t\in [0,T]$, $x\in \mathbb{R}$, $p\in [1,\infty)$, we have $u_\eps(t,x)\in \mathbb{D}^{1,p}$, and the Malliavin derivative satisfies the PDE
        \begin{equation}
          \label{mallDerivAllenCahn}
           \begin{cases}
                \partial_t D_zu_\varepsilon = \frac{1}{2}\Delta u_\varepsilon + \mathfrak{m}D_zu_\varepsilon - \frac{1}{\log\varepsilon^{-1}}f'(t+\varepsilon^2,u_\varepsilon)D_zu_\varepsilon;\\
                D_zu_\varepsilon(0,x) = G_{\varepsilon^2}(x-z).
           \end{cases}
        \end{equation}
    Moreover, if $\varepsilon$ is sufficiently small (the required smallness depending only on $T$), we have the Malliavin derivative bounds
        \begin{equation}
        \label{eq:ptw-Mall-bd}
        0\leq D_zu_\varepsilon(t,x),D_zv_\varepsilon(t,x)\leq e^{3L_1+\mathfrak{m}t}G_{t+\varepsilon^2}(x-z)\qquad\text{for all $t\in [0,T]$ and $x\in\mathbb{R}^2$.}
        \end{equation}    \end{lemma}
    \begin{proof}
     We begin by considering the case when, in addition to the assumptions above, $f$ is \emph{globally} Lipschitz (but we do not impose a quantitative bound on the global Lipschitz constant). In this case, it follows from standard Picard iteration arguments that $u_\eps(t,x)\in \mathbb{D}^{1,p}$ for any $p\in [1,\infty)$, and the PDE (\ref{mallDerivAllenCahn}) is obtained by applying $D_z$ to both sides of the Duhamel formula for $u_\eps$ and then applying the chain rule \cite[Proposition 1.2.3]{MR2200233}. Given the regularity of our initial data in (\ref{mallDerivAllenCahn}) and our assumptions on $f$, it is standard that the mild solution we obtain is in fact a classical solution of (\ref{mallDerivAllenCahn}). 

     Now we will prove the bound (\ref{eq:ptw-Mall-bd}) in this case. The positive initial data in  (\ref{mallDerivAllenCahn}), together with our assumptions that $f_1'$ is bounded and $f_2$ is increasing, implies by the maximum principle that for all $t\in [0,T]$ and $x\in \mathbb{R}^2$, we have
        \begin{equation} D_zu_\varepsilon\ge 0.
        \end{equation}
        Furthermore, these assumptions imply that
        \begin{equation*}
          \begin{split}
                 \partial_t D_zu_\varepsilon &= \frac{1}{2}\Delta D_zu_\varepsilon + \mathfrak{m} D_zu_\varepsilon - \frac{1}{\log\varepsilon^{-1}}f_1'(t+\eps^2,u_\varepsilon)D_zu_\varepsilon -  \frac{1}{\log\varepsilon^{-1}}f_2'(t+\eps^2,u_\varepsilon)D_zu_\varepsilon\\
                 &\leq  \frac{1}{2}\Delta D_zu_\varepsilon + \mathfrak{m} D_zu_\varepsilon + \frac{1}{\log\varepsilon^{-1}}\left(\frac{L_1}{t+\varepsilon^2}\right)D_zu_\varepsilon.
               \end{split}
               \end{equation*}
        Then applying the comparison principle gives
        \begin{equation}\label{eq:Dzalmostdone}
           D_zu_\varepsilon(t,x)\leq\exp\left\{\frac{1}{\log\varepsilon^{-1}}\int_0^t \frac{L_1}{s+\varepsilon^2}\,ds\right\}e^{\mathfrak{m}t}G_{t+\varepsilon^2}(x-z) .
        \end{equation}
        Finally, taking $\varepsilon$ sufficiently small depending only on $T$ implies
        \begin{equation*}
            \exp\left\{\frac{1}{\log\varepsilon^{-1}}\int_0^t \frac{L_1}{s+\varepsilon^2}\,ds\right\} = \exp\left\{\frac{\log\left(\frac{T+\varepsilon^2}{\varepsilon^2}\right)}{\log\varepsilon^{-1}}L_1\right\}\leq e^{3L_1},
        \end{equation*}
      and using this in (\ref{eq:Dzalmostdone}) gives (\ref{eq:ptw-Mall-bd}).

     In the general case when $f$ is not globally Lipschitz, the result can be proved by approximation. The method is fairly standard so we simply sketch the proof. First, we can approximate $f_2$ by a globally Lipschitz element of $\bar S_2$ and show that the resulting approximations $u_{\eps,n}(t,x)$ converge to $u_\eps(t,x)$ in $L^p(\mathbb{P})$ for any $p\in[1,\infty)$. The bound (\ref{mallDerivAllenCahn}) applies to the approximations and shows that $D_zu_{\eps,n}$ is uniformly bounded in $L^p(\mathbb{P})$ as well. Then the fact that $u_\eps(t,x)\in \mathbb{D}^{1,p}$ follows from \cite[Lemma~1.5.3]{MR2200233}. The PDE (\ref{mallDerivAllenCahn}) and resulting bound (\ref{eq:ptw-Mall-bd}) then follows as in the previous case, where we use the more general chain rule \cite[Lemma~4.3]{MR4563698}. The case for $v_\eps$ can be seen either directly from Remark \ref{detMallDer}, or from the fact that the McKean--Vlasov equation is a special case of (\ref{eq:u-f-eqn}). (See Remark \ref{uvSame}.)
    \end{proof}
    It is well-known that an almost sure bound on the $L^2(\mathbb{R}^2)$ norm of the Malliavin derivative implies a sub-Gaussian tail bound on the random variable, which in turn implies moment bounds for all orders. (See \cite{MR2568291} and \cite{MR2329681}.)  By Lemma \ref{ptwMallbd}, we have 
        \begin{equation}
                |Du_\varepsilon(t,x)|_{L^2(\mathbb{R}^2)}^2
                \le \frac{ e^{2|\mathfrak{m}|t+6L_1}}{t+\varepsilon^2}\quad\text{a.s.},
            \end{equation}
         which then implies the following concentration result.
    \begin{lemma}
            \label{concentrationIneq}
            Suppose the nonlinearity $f\in \bar S$ can be written as $f=f_1 +f_2$ with $f_1\in \bar{S}_1(L_1,L_2)$ and $f_2\in \bar{S}_2$. Then, for all $p\in [1,\infty)$ and $T\in [0,\infty)$, we have a constant $K_{p,T}<\infty$ such that 
            \begin{equation}
                \mathbb{P}(|u_\varepsilon(t,x)|\geq \theta) ,\mathbb{P}(|v_\varepsilon(t,x)|\geq \theta)
                \leq 2\exp\left\{-\frac{\theta^2(t+\varepsilon^2)}{2e^{{2|\mathfrak{m}|t}+6L_1}}\right\}
            \end{equation}
            and
            \begin{equation}\label{eq:momentbd}
                    \mathbb{E}[|u_\varepsilon(t,x)|^{p}] \leq \frac{K_{p,T}}{(t+\varepsilon^2)^{p/2}}.
            \end{equation}
            for all $t\in (0,T]$ and 
            $x\in \mathbb{R}^2$.
        \end{lemma}
        The bound (\ref{eq:momentbd}) strengthens (\ref{eq:a-priori-momentbd}): we see now that $K_{p,T}$ need not depend on $\eps$.

\section{Gr\"onwall-type estimates}
\label{GronwallSection}
In this section, we provide the Gr\"onwall-type inequalities which will be used to prove Corollaries~\ref{iterativeMallBd},~\ref{itrLipSolBd}, and~\ref{contractingItrSolBd} below. These estimates are similar to those derived in \cite[Lemmas~4.3 and~5.2]{MR4413215} with slight modifications to fit this context. Throughout this section, the time horizon $T\in (0,\infty)$ and the parameter $\mathfrak{m}\in\mathbb{R}$ are fixed, and we recall the definitions of the time steps $t_m$ and the exponential-scale mesh size $\delta_\eps$ from Section~\ref{RelTimeScale}.
    \begin{lemma}
    \label{Gronwall1}
        Let $t^* \leq t_{m+1}$ and $\theta\in L^\infty([t_m,t_{m+1}]\times \mathbb{R}^2)$. Now suppose that $\theta$ satisfies the following integral inequality on $[t_m,t^*]$:
        \begin{equation}
        \label{G1IntIneq}
            \theta(t,x)\leq \alpha e^{\mathfrak{m}t}G_{t+\varepsilon^2}(x-z) + \frac{\beta}{\log\varepsilon^{-1}}\int_{t_m}^t \frac{1}{s+\varepsilon^2}\int G^\mathfrak{m}_{t-s}(x-y)\theta(s,y)\,dy\,ds.
        \end{equation}
        Then for $\varepsilon$ sufficiently small so that $\delta_\varepsilon < \beta^{-1}e^{-|\mathfrak{m}|T}$  , we have
        \begin{equation}
        \label{eq:Gronwall1}
            \theta(t,x)\leq \frac{\alpha e^{\mathfrak{m}t}}{1-\delta_\varepsilon \beta}G_{t+\varepsilon^2}(x-z)
        \end{equation}
        for all $t\in [t_m,t^*]$ and $x\in \mathbb{R}^2$.
    \end{lemma}
    \begin{proof}
        Suppose that for some $A$ and $B$, we have the bound
        \begin{equation}
        \label{G1IndHyp}
            \theta(t,x)\leq Ae^{\mathfrak{m}t}G_{t+\varepsilon^2}(x-z) + B
        \end{equation}
        for all $t\in [t_m,t^*]$ and $x\in \mathbb{R}^2$. We note that (\ref{G1IndHyp}) holds for $A=0$ and $B = |\theta|_{L^\infty}$. Substituting (\ref{G1IndHyp}) into (\ref{G1IntIneq}) gives
        \begin{equation}
        \label{G1IterBd}
            \begin{split}
                \theta(t,x)&\leq \alpha e^{\mathfrak{m}t}G_{t+\varepsilon^2}(x-z) + \frac{A\beta}{\log\varepsilon^{-1}}\int_{t_m}^t \frac{1}{s+\varepsilon^2}\int G^\mathfrak{m}_{t-s}(x-y)e^{\mathfrak{m}s}G_{s+\varepsilon^2}(y-z)\,dy\,ds\\
                &\quad + \frac{B\beta}{\log\varepsilon^{-1}}\int_{t_m}^t \frac{1}{s+\varepsilon^2}\int G^\mathfrak{m}_{t-s}(x-y)\,dy\,ds\\
                &\leq \alpha e^{\mathfrak{m}t}G_{t+\varepsilon^2}(x-z) + \frac{A\beta e^{\mathfrak{m}t}}{\log\varepsilon^{-1}}G_{t+\varepsilon^2}(x-z)\int_{t_m}^{t_m+1}\frac{1}{s+\varepsilon^2}\,ds
                \\&\quad+ \frac{B\beta e^{|\mathfrak{m}|T}}{\log\varepsilon^{-1}}\int_{t_m}^{t_{m+1}}\frac{1}{s+\varepsilon^2}\,ds\\
                &= (\alpha + A\beta \delta_\varepsilon)e^{\mathfrak{m}t}G_{t+\varepsilon^2}(x-z) + B\beta e^{|\mathfrak{m}|T}\delta_\varepsilon.
            \end{split}
        \end{equation}
        Now, set $A^{(0)}=0$, $B^{(0)}=|\theta|_{L^\infty}$, and inductively define $A^{(n+1)}= \alpha + A^{(n)}\beta\delta_\varepsilon$ and $B^{(n+1)}=B^{(n)}\beta e^{|\mathfrak{m}|T}\delta_\varepsilon$. Then (\ref{G1IterBd}) implies that for all $n\in \mathbb{N}$, we have
        \begin{equation}
            \theta(t,x)\leq \alpha e^{\mathfrak{m}t}G_{t+\varepsilon^2}(x-z)\sum_{j=0}^{n-1}(\beta\delta_\varepsilon)^{j}  + |\theta|_{L^\infty}(\beta e^{|\mathfrak{m}|T}\delta_\varepsilon)^n.
        \end{equation}
        Taking $\varepsilon$ sufficiently small so that $\delta_\varepsilon < \beta^{-1}e^{-|\mathfrak{m}|T}$ and then taking $n\to \infty$ gives (\ref{eq:Gronwall1}).
    \end{proof}

    \begin{lemma}
    \label{Gronwall2}
        Suppose that $\theta$ is continuous and satisfies the following integral inequality on $[t_m,t_{m+1}]$:
        \begin{equation}
        \label{G2Intineq}
            \theta(t)\leq \frac{\alpha}{\sqrt{t+\varepsilon^2}} + \frac{\beta}{(\log\eps^{-1})\sqrt{t+\varepsilon^2}}\int_{t_m}^t \frac{\theta(s)}{\sqrt{s+\varepsilon^2}}\,ds.
        \end{equation}
        Then for $\varepsilon$ sufficiently small so that $\delta_\varepsilon < \beta^{-1}$, we have the bound
        \begin{equation}
        \label{eq:Gronwall2}
        \theta(t)\leq\frac{\alpha}{(1-\beta\delta_\varepsilon)\sqrt{t+\varepsilon^2}}\qquad\text{for all $t\in [t_m,t_{m+1}]$}.
        \end{equation}
    \end{lemma}
    \begin{proof}
      Let
      \begin{equation}
        A\coloneqq \sup_{t\in [t_m,t_{m+1}]} (t+\eps^2)^{1/2}\theta(t),
      \end{equation}
      which is certainly finite since $\theta$ is continuous.
      Then (\ref{G2Intineq}) tells us that, for all $t\in [t_m,t_{m+1}]$, we have
      \[
        (t+\eps^2)^{1/2}\theta(t) \le \alpha + \frac{\beta A}{(\log\eps^{-1})}\int_{t_m}^t \frac{ds}{s+\eps^2} \le\alpha + \frac{\beta A}{(\log\eps^{-1})}\int_{t_m}^{t_{m+1}} \frac{ds}{s+\eps^2} \overset{(\ref{eq:reasonfortms})}= \alpha+\beta A\delta_\eps.
      \]
      This implies that $A\le \alpha +\beta A\delta_\eps$, so we obtain $A\le \alpha/(1-\beta\delta_\eps)$ and hence (\ref{Gronwall2}) whenever $\delta_\eps\beta<1$.
    \end{proof}

       \section{Proof of Proposition \ref{lipSoltnBd}}
       \label{proof-lipschitz}
       In this section, we will prove Proposition \ref{lipSoltnBd}, which gives (\ref{eq:main-theorem-McK-V}) under the stronger assumption that the nonlinearity $f$ satisfies the $C^{1,1}$ bounds (\ref{S1Bd}). In Section~\ref{mainResultSection}, we will extend this to general $f\in \bar S$ by approximation. First, we will decompose the equation for the error in a way that isolates the ``leading order Gaussian behavior.'' In order to maintain good control over the error, we will focus on each short interval $[t_m,t_{m+1}]$ individually. (See (\ref{eq:timescale}).) That is, we will construct an iterative scheme where at each step, we verify that the error has has not grown too large over the interval $[t_m,t_{m+1}]$, and then use this estimate to bound the amount of error generated over the next interval $[t_{m+1},t_{m+2}]$. In doing so, we are able to bound the error at the final time $T$ and obtain the quantitative rate of convergence (\ref{eq:lip-soltn-bd}).

       Throughout this section, we work with a fixed time horizon $T_0\in (0,\infty)$ and an arbitrary $T\in (0,T_0]$. Constants will depend on $T_0$ but not on $T$, and we assume that
       \begin{equation}\label{fS1}
         f\in \bar{S}_1(L_1,L_2) \qquad\text{for some $L_1,L_2\ge 0$}.
       \end{equation}

      \subsection{Bounding the Malliavin derivative on a short interval}
      The following estimate is the key to our iteration scheme. Let \[w_\varepsilon := u_\varepsilon - v_\varepsilon.\] We will first prove in Proposition~\ref{lipMainEstimate} an integral inequality for $|D_zw_\varepsilon|_{L^2(\mathbb{P})}$ in terms of $|D_zw_\eps|_{L^2(\mathbb{P})}$ and $|w_\varepsilon|_{L^2(\mathbb{P})}$. In Corollary~\ref{iterativeMallBd}, we eliminate the dependence on $|D_zw_\eps|_{L^2(\mathbb{P})}$ via the Grönwall-type inequality Lemma~\ref{Gronwall1}. In Corollary~\ref{itrLipSolBd}, we use the Gaussian Poincaré inequality to turn this into a bound on $|w_\eps|_{L^2(\mathbb{P})}$ in terms of $|w_\eps|_{L^2(\mathbb{P})}$ again, and then eliminate the last dependence using the Grönwall-type inequality Lemma~\ref{Gronwall2}. Substituting this back into the result of Corollary~\ref{iterativeMallBd}, we get an unconditional bound on $|D_zw_\eps|_{L^2(\mathbb{P})}$ on a short time interval. This will be the basis of our induction in Section~\ref{itrSection}.
     \begin{proposition}
    \label{lipMainEstimate} 
    Define $w_\eps$ as above and fix $z\in\mathbb{R}^2$. 
    Suppose that there is an $A<\infty$ such that
        \begin{equation}
        \label{mallBoundAssumption}
        |D_zw_\varepsilon(t_m,x)|_{L^2(\mathbb{P})}\leq \frac{A}{(\log\varepsilon^{-1})^{1/4}}e^{\mathfrak{m}t_m}G_{t_m+\varepsilon^2}(x-z)\qquad\text{for all $x\in\mathbb{R}^2$.}
        \end{equation}
    Then for all $t\in [t_m,t_{m+1}]$ and $x\in \mathbb{R}^2$, we have the bound
  \begin{equation}
          \begin{split}
              |D_zw_\varepsilon(t,x)|_{L^2(\mathbb{P})}&\leq \frac{A +  L_2 e^{3L_1+|\mathfrak{m}|T_0}\delta_\varepsilon}{(\log\varepsilon^{-1})^{1/4}}e^{\mathfrak{m}t}G_{t+\varepsilon^2}(x-z)\\
        &\qquad +\frac{L_1}{\log\varepsilon^{-1}}\int_{t_m}^t\frac{1}{s+\varepsilon^2}\int G^\mathfrak{m}_{t-s}(x-y)|D_zw_\varepsilon(s,y)|_{L^2(\mathbb{P})}\,dy\,ds\\
         &\qquad + \frac{ e^{3L_1+\mathfrak m t }L_2}{\log\varepsilon^{-1}}G_{t+\varepsilon^2}(x-z)\int_{t_m}^t \frac{|w_\varepsilon(s,x)|_{L^2(\mathbb{P})}}{\sqrt{s+\varepsilon^2}}\,ds.
            \end{split}\label{eq:Dzwbound_1}
      \end{equation}
    \end{proposition}
    \begin{proof}
        We have the following PDE for $D_zw_\eps$:
        \begin{equation}
            \begin{split}
                \partial_t D_zw_\varepsilon &= \frac{1}{2}\Delta D_zw_\varepsilon + \mathfrak{m} D_zw_\varepsilon - \frac{1}{\log\varepsilon^{-1}}f'(t+\eps^2,u_\varepsilon)D_zu_\varepsilon + \frac{1}{\log\varepsilon^{-1}}\mathbb{E}[f'(t+\eps^2,v_\varepsilon)]D_zv_\varepsilon\\
                &= \frac{1}{2}\Delta D_zw_\varepsilon + \mathfrak{m} D_zw_\varepsilon - \frac{1}{\log\varepsilon^{-1}}f'(t+\eps^2,u_\varepsilon)D_zw_\varepsilon
                \\&\qquad+ \frac{1}{\log\varepsilon^{-1}}(\mathbb{E}[f'(t+\eps^2,v_\varepsilon)]-f'(t+\eps^2,v_\varepsilon))D_zv_\varepsilon\\
                &\qquad + \frac{1}{\log\varepsilon^{-1}}(f'(t+\eps^2,v_\varepsilon)-f'(t+\eps^2,u_\varepsilon))D_zv_\varepsilon.
            \end{split}
        \end{equation}
        Applying Duhamel's formula, we have for all $t\in [t_m,t_{m+1}]$ and $x\in \mathbb{R}^2$ that
        \begin{align}
            D_z&w_\varepsilon(t,x) \notag\\&= \int G^\mathfrak{m}_{t-t_m}(x-y)D_zw_\varepsilon(t_m,y)\,dy\notag\\
                 &\qquad - \frac{1}{\log\varepsilon^{-1}}\int_{t_m}^t\int G^\mathfrak{m}_{t-s}(x-y)f'(s+\eps^2,u_\varepsilon(s,y))D_zw_\varepsilon(s,y)\,dy\,ds\notag\\
                 &\qquad +\frac{1}{\log\varepsilon^{-1}}\int_{t_m}^t\int G^\mathfrak{m}_{t-s}(x-y)
                (\mathbb{E}[f'(s+\eps^2,v_\varepsilon(s,y)]-f'(s+\eps^2,v_\varepsilon(s,y)))D_zv_\varepsilon(s,y)\,dy\,ds\notag\\
                 &\qquad + \frac{1}{\log\varepsilon^{-1}}\int_{t_m}^t\int G^\mathfrak{m}_{t-s}(x-y)
                (f'(s+\eps^2,v_\varepsilon(s,y))-f'(s+\eps^2,u_\varepsilon(s,y)))D_zv_\varepsilon(s,y)\,dy\,ds\notag\\
                   &\eqqcolon I_0+I_1+I_2+I_3.\label{eq:I0I1I2I3}
              \end{align}
        We will bound each term separately in the $|\cdot|_{L^2(\mathbb{P})}$ norm.

        For  $I_0$, applying (\ref{mallBoundAssumption}) gives
        \begin{align}
                \notag |I_0|_{L^2(\mathbb{P})} &\leq \int G^\mathfrak{m}_{t-t_m}(x-y)|D_zw_\varepsilon(t_m,y)|_{L^2(\mathbb{P})}\,dy\\
                \notag &\leq \frac{A}{(\log\varepsilon^{-1})^{1/4}}\int G_{t-t_m}^\mathfrak{m}(x-y)e^{\mathfrak{m}t_m}G_{t_m+\varepsilon^2}(y-z)\,dy\\
                       &= \frac{A}{(\log\varepsilon^{-1})^{1/4}}e^{\mathfrak{m}t}G_{t+\varepsilon^2}(x-z).\label{eq:I0bd}
              \end{align}

        For $I_1$, our derivative bound (\ref{S1Bd}) implies that
        \begin{align}
                \notag |I_1|_{L^2(\mathbb{P})}&\leq \frac{1}{\log\varepsilon^{-1}}\int_{t_m}^t\int G^\mathfrak{m}_{t-s}(x-y)|f'(s+\eps^2,u_\varepsilon(s,u))D_zw_\varepsilon(s,y)|_{L^2(\mathbb{P})}\,dy\,ds\\
                                              &\leq \frac{L_1}{\log\varepsilon^{-1}}\int_{t_m}^t \frac{1}{s+\varepsilon^2}\int G^\mathfrak{m}_{t-s}(x-y)|D_zw_\varepsilon(s,y)|_{L^2(\mathbb{P})}\,dy\,ds. \label{eq:I1bd}
              \end{align}

        Before bounding $I_2$, we first observe that by the Gaussian Poincar\'e covariance inequality (see e.g.~\cite[Lemma 4.2]{MR4612705}), we have
        \begin{align}
          \notag &\hspace{-3em}\left|\mathbb{E}\left[\prod_{j=1}^2\Big{(}f'(s_j+\eps^2,v_\varepsilon(s_j,y_j))-\mathbb{E}[f'(s_j+\eps^2,v_\varepsilon(s_j,y_j))]\Big{)}\right]\right|\\
          \notag &= \left|\Cov\left(f'(s_1+\eps^2,v_\varepsilon(s_1,y_1)),f'(s_2+\eps^2,v_\varepsilon(s_2,y_2))\right)\right|\\
          \notag &\leq \int |D_z[f'(s_1+\eps^2,v_\varepsilon(s_1,y_1))]|_{L^2(\mathbb{P})}|D_z[f'(s_2+\eps^2,v_\varepsilon(s_2,y_2))]|_{L^2(\mathbb{P})}\,dz.\\
        \notag \overset{(\ref{S1Bd})}&\leq L_2^2\int \frac{|D_zv_\varepsilon(s_1,y_1)|_{L^2(\mathbb{P})}|D_zv_\varepsilon(s_2,y_2)|_{L^2(\mathbb{P})}}{(s_1+\varepsilon^2)^{1/2}(s_2+\varepsilon^2)^{1/2}} \,dz\\
                \notag \overset{(\ref{eq:ptw-Mall-bd})}&\leq e^{6L_1}L_2^2\int \frac{e^{\mathfrak{m}s_1}G_{s_1+\varepsilon^2}(y_1-z) e^{\mathfrak{m}s_2}G^\mathfrak{m}_{s_2+\varepsilon^2}(z-y_2)}{(s_1+\varepsilon^2)^{1/2}(s_2+\varepsilon^2)^{1/2}}\,dz\\
                &= e^{6L_1}L_2^2 \frac{ e^{\mathfrak{m}(s_1+s_2)}G_{s_1+s_2+2\varepsilon^2}(y_1-y_2)}{(s_1+\varepsilon^2)^{1/2}(s_2+\varepsilon^2)^{1/2}}.
        \label{gaussCov}
              \end{align}
        Returning to $I_2$, we see that
        \begin{align}
               \mathbb{E}[I_2^2]
               &
               \notag = \mathbb{E}\Big{[}\Big{(}\frac{1}{\log\varepsilon^{-1}}\int_{t_m}^t\int G_{t-s}(x-y)\big{(}\mathbb{E}[f'(s+\eps^2,v_\varepsilon(s,y))]
               \\\notag&\hspace{20em}-f'(s+\eps^2,v_\varepsilon(s,y))\big{)}D_zv_\varepsilon(s,y)\,dy\,ds\Big{)}^2\Big{]}\\
               \notag & = \int_{t_m}^t\int_{t_m}^t\iint \left(\prod_{j=1}^2\frac{G^\mathfrak{m}_{t-s_j}(x-y_j)D_zv_\varepsilon(s_j,y_j)}{\log\varepsilon^{-1}}\right)
               \\\notag&\hspace{7em}\cdot \mathbb{E}\left[\prod_{j=1}^2\left(f'(s_j+\eps^2,v_\varepsilon(s_j,y_j))-\mathbb{E}[f'(s_j+\eps^2,v_\varepsilon(s_j,y_j))]\right)\right]\,dy_1\,dy_2\,ds_1\,ds_2\\
               \notag &\leq \int_{t_m}^t\int_{t_m}^t \frac{L_2^2e^{12L_1+2|\mathfrak{m}|T_0}}{(s_1+\varepsilon^2)^{1/2}(s_2+\varepsilon^2)^{1/2}}
                     \\&\hspace{4em}\cdot\iint \left(\prod_{j=1}^2\frac{G^\mathfrak{m}_{t-s_j}(x-y_j)e^{\mathfrak{m}s_j}G_{s_j+\varepsilon^2}(y_j-z)}{\log\varepsilon^{-1}}\right)G_{s_1+s_2+2\varepsilon^2}(y_1-y_2)\,dy_1\,dy_2\,ds_1\,ds_2,
       \label{I2ineq1}
       \end{align}
       where in the second line we used the fact that $D_zv_\varepsilon$ is deterministic as noted in Remark~\ref{detMallDer}, and in the last line we used (\ref{eq:ptw-Mall-bd}) and (\ref{gaussCov}). 
       Next, we see that
       \begin{align*}
           \notag\iint \left(\prod_{j=1}^2\frac{G^\mathfrak{m}_{t-s_j}(x-y_j)e^{\mathfrak{m}s_j}G_{s_j+\varepsilon^2}(y_j-z)}{\log\varepsilon^{-1}}\right)G_{s_1+s_2+2\varepsilon^2}(y_1-y_2)\,dy_1\,dy_2 
           \\\notag&\hspace{-25em}\leq \frac{1}{s_1+s_2+2\varepsilon^2}\iint \left(\prod_{j=1}^2\frac{G^\mathfrak{m}_{t-s_j}(x-y_j)e^{\mathfrak{m}s_j}G_{s_j+\varepsilon^2}(y_j-z)}{\log\varepsilon^{-1}}\right)\,dy_1\,dy_2 
           \\&\hspace{-25em}= \frac{e^{2\mathfrak{m}t}G_{t+\varepsilon^2}(x-z)^2}{2\pi(s_1+s_2+2\varepsilon^2)(\log\varepsilon^{-1})^2}.
       \end{align*}
       Using this bound in (\ref{I2ineq1}) gives
      \begin{equation}
          \begin{split}
                \mathbb{E}[I_2^2] &\leq  \frac{L_2^2e^{12L_1+2|\mathfrak{m}|T_0}e^{2\mathfrak{m}t}G_{t+\varepsilon^2}(x-z)^2}{2\pi(\log\varepsilon^{-1})^2}
                \\&\hspace{10em}\cdot\int_{t_m}^t \frac{1}{(s_2+\varepsilon^2)^{1/2}} \int_{t_m}^t \frac{1}{(s_1+\varepsilon^2)^{1/2}(s_1+s_2+2\varepsilon^2)}\,ds_1\,ds_2 
           \\&= \frac{L_2^2e^{12L_1+2|\mathfrak{m}|T_0}e^{2\mathfrak{m}t}G_{t+\varepsilon^2}(x-z)^2}{\pi(\log\varepsilon^{-1})^2}\int_{t_m}^t \frac{1}{(s_2+\varepsilon^2)^{1/2}}\int_{\sqrt{t_m+\varepsilon^{2}}}^{\sqrt{t+\varepsilon^2}} \frac{1}{r^2+s_2+\varepsilon^2}\,dr\,ds_2
           \\&= \frac{L_2^2e^{12L_1+2|\mathfrak{m}|T_0}e^{2\mathfrak{m}t}G_{t+\varepsilon^2}(x-z)^2}{2\pi(\log\varepsilon^{-1})^2}\int_{t_m}^t
           \frac{\text{arctan}\left(\frac{\sqrt{t+\varepsilon^2}}{\sqrt{s_2+\varepsilon^2}}\right)-\text{arctan}\left(\frac{\sqrt{t_m+\varepsilon^2}}{\sqrt{s_2+\varepsilon^2}}\right)}{s_2+\varepsilon^2}\,ds_2.
          \end{split}
      \end{equation}
      Then using the bound $|\text{arctan}|\leq \frac{\pi}{2}$, we have
      \begin{equation}
      \label{I2ineq2}
          \begin{split}
              \mathbb{E}[I_2^2] &\leq \frac{L_2^2e^{12L_1+2|\mathfrak{m}|T_0}e^{2\mathfrak{m}t}G_{t+\varepsilon^2}(x-z)^2}{2(\log\varepsilon^{-1})^2}\int_{t_m}^{t_{m+1}}\frac{1}{s_2+\varepsilon^2}\,ds_2\\
              \overset{(\ref{eq:reasonfortms})}&= \frac{L_2^2e^{12L_1+2|\mathfrak{m}|T_0}\delta_\varepsilon e^{2\mathfrak{m}t}G_{t+\varepsilon^2}(x-z)^2}{2\log\varepsilon^{-1}}.
          \end{split}
      \end{equation}
      Because we will be iterating our error bound approximately $\delta_\varepsilon^{-1}$ times, hence accruing error of order $\delta_{\eps}^{-1}(\mathbb{E}[I_2^2])^{1/2}$, we will need to rewrite the right side of (\ref{I2ineq2}) to contain a term of order at least $\delta_\varepsilon^2$
      for the error to be summable. This means that we need to sacrifice some of our rate of convergence and, using~(\ref{eq:deltaMepsdefs}) rewrite (\ref{I2ineq2}) as
      \begin{equation}
      \label{rewriteStep}
           \mathbb{E}[I_2^2] \leq \frac{L_2^2e^{12L_1+2|\mathfrak{m}|T_0}\delta_\varepsilon^2 e^{2\mathfrak{m}t}G_{t+\varepsilon^2}(x-z)^2}{2(\log\varepsilon^{-1})^{1/2}}.
      \end{equation}
      We remark that although there is some flexibility in the choice of $\delta_\varepsilon$, this places a restriction on how small we are allowed to choose our time steps in Section~\ref{RelTimeScale}.

      For the final term $I_3$, we apply our $C^{1,1}$ bound (\ref{S1Bd}) as well as the spatial stationarity of $w_\eps$ to obtain
      \begin{equation}
          \begin{split}
              |I_3|_{L^2(\mathbb{P})} &\leq \frac{ e^{3L_1}}{\log\varepsilon^{-1}}\int_{t_m}^t |f'(t+\eps^2,v_\varepsilon(s,x))-f'(t+\eps^2,u_\varepsilon(s,x))|_{L^2(\mathbb{P})}
              \\&\hspace{15em}\cdot\int G^\mathfrak{m}_{t-s}(x-y)e^{\mathfrak{m}s}G_{s+\varepsilon^2}(y-z)\,dy\,ds\\
              &= \frac{ e^{3L_1}L_2}{\log\varepsilon^{-1}}e^{\mathfrak{m}t}G_{t+\varepsilon^2}(x-z)\int_{t_m}^t \frac{|w_\varepsilon(s,x)|_{L^2(\mathbb{P})}}{\sqrt{s+\varepsilon^2}}\,ds.
            \end{split}\label{eq:I3bd}
      \end{equation}
      
      Now (\ref{eq:I0bd}), (\ref{eq:I1bd}), (\ref{rewriteStep}), and (\ref{eq:I3bd}) to bound $I_0$, $I_1$, $I_2$, and $I_3$, respectively, in (\ref{eq:I0I1I2I3}), we conclude that
      \begin{equation}
          \begin{split}
              |D_zw_\varepsilon(t,x)|_{L^2(\mathbb{P})}&\leq |I_0|_{L^2(\mathbb{P})}+|I_1|_{L^2(\mathbb{P})}+|I_2|_{L^2(\mathbb{P})}+|I_3|_{L^2(\mathbb{P})}\\
              &\leq \frac{A +  L_2 e^{6L_1+|\mathfrak{m}|T_0}\delta_\varepsilon}{(\log\varepsilon^{-1})^{1/4}}e^{\mathfrak{m}t}G_{t+\varepsilon^2}(x-z)\\
        &\qquad +\frac{L_1}{\log\varepsilon^{-1}}\int_{t_m}^t\frac{1}{(s+\varepsilon^2)}\int G^\mathfrak{m}_{t-s}(x-y)|D_zw_\varepsilon(s,y)|_{L^2(\mathbb{P})}\,dy\,ds\\
         &\qquad + \frac{ e^{3L_1 + \mathfrak m t }L_2}{\log\varepsilon^{-1}}G_{t+\varepsilon^2}(x-z)\int_{t_m}^t \frac{|w_\varepsilon(s,x)|_{L^2(\mathbb{P})}}{\sqrt{s+\varepsilon^2}}\,ds,
          \end{split}
      \end{equation}
      which completes the proof.
    \end{proof}

    The second term on the right side of (\ref{eq:Dzwbound_1}) can be treated using the Grönwall-type inequality Lemma~\ref{Gronwall1}.
    \begin{corollary}
    \label{iterativeMallBd}
   Under the assumption (\ref{mallBoundAssumption}), for all $\eps$ sufficiently small so that
   \begin{equation}
   \label{eq:delta-eps-condition-1}
       \delta_\eps < (e^{|\mathfrak{m}|T_0}L_1)^{-1},
   \end{equation}
   we have for all $t\in [t_m,t_{m+1}]$ and $x\in \mathbb{R}^2$ that
        \begin{multline}
          \label{eq:iterativeMallBd}|D_zw_\varepsilon(t,x)|_{L^2(\mathbb{P})}\\\leq \left(\frac{A + L_2 e^{6L_1+|\mathfrak{m}|T_0}\delta_\varepsilon}{(\log\varepsilon^{-1})^{1/4}}  + \frac{ e^{3L_1 }L_2}{\log\varepsilon^{-1}}\int_{t_m}^t \frac{|w_\varepsilon(s,x)|_{L^2(\mathbb{P})}}{\sqrt{s+\varepsilon^2}}\,ds\right)\frac{e^{\mathfrak{m}t}}{1-L_1\delta_\varepsilon}G_{t+\varepsilon^2}(x-z).
        \end{multline}
    \end{corollary}
    \begin{proof}
By Lemma~\ref{ptwMallbd} and the triangle inequality, we have
        \begin{equation}
            |D_zw_\varepsilon(t,x)|_{L^2(\mathbb{P})}\leq 2e^{3L_1+|\mathfrak{m}|T_0}e^{\mathfrak{m}t}G_{t+\varepsilon^2}(x-z),
        \end{equation}
        so $ |D_zw_\varepsilon(t,x)|_{L^2(\mathbb{P})}\in L^\infty([t_m,t_{m+1}]\times \mathbb{R}^2)$. Then (\ref{eq:iterativeMallBd}) follows immediately from Lemma \ref{Gronwall1}, where we consider the interval $[t_m,t]$ and set 
        \begin{equation}
            \alpha = \left(\frac{A +  L_2 e^{6L_1+|\mathfrak{m}|T_0}\delta_\varepsilon }{(\log\varepsilon^{-1})^{1/4}} + \frac{ e^{3L_1 }L_2}{\log\varepsilon^{-1}}\int_{t_m}^t \frac{|w_\varepsilon(s,x)|_{L^2(\mathbb{P})}}{\sqrt{s+\varepsilon^2}}\,ds\right)e^{\mathfrak{m}t}G_{t+\varepsilon^2}(x-z)
        \end{equation}
        and $\beta = L_1$, noting that with these choices, the assumption (\ref{G1IntIneq}) is satisfied by (\ref{eq:Dzwbound_1}).
    \end{proof}
    The next step will be to translate Lemma~\ref{iterativeMallBd} into a bound on $|w_\varepsilon|_{L^2(\mathbb{P})}$ over the interval $[t_m,t_{m+1}]$ using the Gaussian Poincar\'e inequality. 
    \begin{corollary}
    \label{itrLipSolBd}
   Under the assumption (\ref{mallBoundAssumption}), for any $\eps>0$ such that (\ref{eq:delta-eps-condition-1}) holds and
    \begin{equation}
    \label{eq:delta-eps-condition-2}
        \delta_\eps \leq \left(\frac{e^{3L_1}L_2}{1-L_1\delta_\eps}\right)^{-1},
    \end{equation}
        we have, for all $t\in [t_m,t_{m+1}]$ and $x\in \mathbb{R}^2$, that 
        \begin{equation}
        \label{itrLipSolBdResult}
        |w_\varepsilon(t,x)|_{L^2(\mathbb{P})}\leq \frac{e^{|\mathfrak{m}|t_{m+1}}(A+L_2e^{6L_1+|\mathfrak{m}|T_0}\delta_\eps)}{(1-(L_1+e^{3L_1}L_2)\delta_\eps)(\log\eps^{-1})^{1/4}\sqrt{t+\varepsilon^2}}.
        \end{equation}
    \end{corollary}
    \begin{proof}
      Applying the Gaussian Poincar\'e inequality (see e.g.\  \cite[Proposition 3.1]{MR2527030}) to $w_\varepsilon$ and using Corollary~\ref{iterativeMallBd}, we obtain for all $t\in [t_m,t_{m+1}]$ and $x\in \mathbb{R}^2$ that
        \begin{align}
          |w_\eps&(t,x)|_{L^2(\mathbb{P})} \leq \left(\int \mathbb{E}[D_zw_\varepsilon(t,x)^2]\,dz\right)^{1/2}\notag\\
                &\leq  (1-L_1\delta_\varepsilon)^{-1}\left(\frac{A +  L_2 e^{3L_1+|\mathfrak{m}|T_0}\delta_\varepsilon}{(\log\varepsilon^{-1})^{1/4}}  + \frac{ e^{3L_1 }L_2}{\log\varepsilon^{-1}}\int_{t_m}^t \frac{|w_\varepsilon(s,x)|_{L^2(\mathbb{P})}}{\sqrt{s+\varepsilon^2}}\,ds\right)
          \notag\\&\hspace{15em}\cdot\left(\int e^{\mathfrak{m}t}G_{t+\varepsilon^2}(x-z)e^{\mathfrak{m}t}G_{t+\varepsilon^2}(z-x)\,dz\right)^{1/2}\notag\\
                &\leq\frac{e^{|\mathfrak{m}|t_{m+1}}}{(1-L_1\delta_\eps)\sqrt{t+\varepsilon^2}}\left(\frac{A + L_2 e^{6L_1+|\mathfrak{m}|T_0}\delta_\varepsilon}{(\log\varepsilon^{-1})^{1/4}}  + \frac{ e^{3L_1 }L_2}{\log\varepsilon^{-1}}\int_{t_m}^t \frac{|w_\varepsilon(s,x)|_{L^2(\mathbb{P})}}{\sqrt{s+\varepsilon^2}}\,ds\right).
              \end{align}
            In the last inequality we used the fact that $\int G_{t+\eps^2}(x-z)^2\,dz=\frac{1}{4\pi (t+\eps^2)}\le \frac1{t+\eps^2}$. Then applying  Lemma~\ref{Gronwall2} with
        \begin{equation}
            \alpha = e^{|\mathfrak{m}|t_{m+1}}\frac{A +  L_2 e^{6L_1+|\mathfrak{m}|T_0}\delta_\varepsilon}{(1-L_1\delta_\varepsilon)(\log\varepsilon^{-1})^{1/4}}\quad\text{and}\quad
        \beta = \frac{ e^{3L_1}L_2}{1-L_1\delta_\varepsilon}
    \end{equation}
    gives (\ref{itrLipSolBdResult}).
    \end{proof}
The final step is to substitute the bound obtained in Corollary \ref{itrLipSolBd} back into (\ref{itrLipSolBdResult}) in order to bound $|D_zw_\varepsilon(t_{m+1},\cdot)|_{L^2(\mathbb{P})}$, so we are ready to move on to the next interval.
    \begin{corollary}
    \label{KmInduciton}
          For any $\eps>0$ such that (\ref{eq:delta-eps-condition-1}) and (\ref{eq:delta-eps-condition-2}) both hold,
           if we have the bound
        \begin{equation}
        \label{KmIndHyp}
        |D_zw_\varepsilon(t_m,x)|_{L^2(\mathbb{P})} \leq \frac{A_m}{(\log\varepsilon^{-1})^{1/4}}e^{\mathfrak{m}t_m}G_{t_m+\varepsilon^2}(x-z)\qquad\text{for all }x\in\mathbb{R}^2,
        \end{equation}
        (i.e. (\ref{mallBoundAssumption}) with $A$ now denoted by $A_m$),
        then
        \begin{equation}
          \label{eq:Km-induction-result}|D_zw_\varepsilon(t_{m+1},x)|_{L^2(\mathbb{P})} \leq \frac{A_{m+1}}{(\log\varepsilon^{-1})^{1/4}}e^{\mathfrak{m}t_{m+1}}G_{t_{m+1}+\varepsilon^2}(x-z)\qquad\text{for all }x\in\mathbb{R}^2,
        \end{equation}
        where
        \begin{equation}
            A_{m+1} = \left(A_m + L_2 e^{6L_1 + |\mathfrak{m}|T_0}\delta_\varepsilon\right)\left(1+ \frac{e^{3L_1 + |\mathfrak{m}|T_0}L_2\delta_\varepsilon}{1-(L_1+e^{3L_1}L_2)\delta_\eps} \right)(1-L_1\delta_\varepsilon)^{-1}.
        \end{equation}
        \begin{proof}
            To simplify notation, we set
            \begin{equation}
              \rho_1(\varepsilon) \coloneqq (1-L_1\delta_\varepsilon)^{-1}\quad\text{and}\quad \rho_2(\varepsilon) \coloneqq (1-(L_1+e^{3L_1}L_2)\delta_\eps)^{-1}.
            \end{equation}
            By our induction hypothesis (\ref{KmIndHyp}), the conditions of Corollaries~\ref{iterativeMallBd} and~\ref{itrLipSolBd} are satisfied and we obtain for sufficiently small $\varepsilon$ depending on $L_1$ and $L_2$ that
           \begin{equation*}
               \begin{split}
                    |D_zw_\varepsilon(t,x)|_{L^2(\mathbb{P})}
                     &\leq  \rho_1(\varepsilon)\left(\frac{A_m + L_2 e^{6L_1+|\mathfrak{m}|T_0}\delta_\varepsilon}{(\log\varepsilon^{-1})^{1/4}}\right)  e^{\mathfrak{m}t}G_{t+\varepsilon^2}(x-z)
                     \\&\qquad + \rho_1(\varepsilon)\frac{ e^{3L_1 }L_2}{\log\varepsilon^{-1}}e^{\mathfrak{m}t}G_{t+\varepsilon^2}(x-z)\int_{t_m}^t \frac{|w_\varepsilon(s,x)|_{L^2(\mathbb{P})}}{\sqrt{s+\varepsilon^2}}\,ds\\
                     &\leq  \rho_1(\varepsilon)\left(\frac{A_m + L_2 e^{6L_1+|\mathfrak{m}|T_0}\delta_\varepsilon}{(\log\varepsilon^{-1})^{1/4}}\right) \delta_\varepsilon e^{\mathfrak{m}t}G_{t+\varepsilon^2}(x-z)
                     \\&\qquad + \left(\frac{A_m +  L_2 e^{6L_1+|\mathfrak{m}|T_0}\delta_\varepsilon}{(\log\varepsilon^{-1})^{1/4}}\right)\frac{ e^{3L_1 }L_2 }{\log\varepsilon^{-1}}e^{\mathfrak{m}t}G_{t+\varepsilon^2}(x-z)\int_{t_m}^t \frac{\rho_1(\varepsilon)\rho_2(\varepsilon)}{s+\varepsilon^2}\,ds\\
                     &\leq \left(\frac{A_m+ L_2e^{6L_1+|\mathfrak{m}|T_0}\delta_\varepsilon}{(\log\varepsilon^{-1})^{-1/4}}\right)\Bigl{(}1 + \rho_2(\eps) e^{3L_1+|\mathfrak{m}|}L_2\delta_\varepsilon\Bigr{)}\rho_1(\varepsilon) e^{\mathfrak{m}t}G_{t+\varepsilon^2}(x-z),
               \end{split}
           \end{equation*}
           which is (\ref{eq:Km-induction-result}).
        \end{proof}
    \end{corollary}
    \subsection{Inductive bound on the solution}
    \label{itrSection}
    In this section, we will complete the proof of Proposition~\ref{lipSoltnBd} by iterating Corollary~\ref{KmInduciton} over $m$ until we obtain the amount of error accumulated over the entire interval $[0,T]$. By our assumption (\ref{eq:delta-eps-condition}), the conditions of Corollary~\ref{KmInduciton} are satisfied for $m=0$ immediately from the fact that $w_\varepsilon(0,\cdot) = 0$. Then by induction, Corollary \ref{KmInduciton} implies that for all $0 \leq m\leq M_\varepsilon$, we have the bound
    \begin{equation}
    \label{lipItrBdResult}
        |D_zw_\varepsilon(t_m,x)|_{L^2(\mathbb{P})} \leq \frac{A_{m}}{(\log\eps^{-1})^{1/4}}e^{\mathfrak{m}t_{m}}G_{t_m+\eps^2}(x-z),
    \end{equation}
    where $A_0 = 0$ and $A_{m+1}(\varepsilon)$ is defined inductively by
    \begin{equation}
           A_{m+1}(\varepsilon) = \left(A_m(\varepsilon) +  L_2 e^{6L_1 + |\mathfrak{m}|T_0}\delta_\varepsilon\right)(1-L_1\delta_\varepsilon)^{-1}\left(1+ \frac{e^{3L_1 + |\mathfrak{m}|T_0}L_2\delta_\varepsilon}{1-(L_1+e^{3L_1}L_2)\delta_\eps} \right).
    \end{equation}
    In particular, we have that
   \begin{equation}
       \begin{split}
              A_{m+1}(\varepsilon) &=  L_2 e^{6L_1+|\mathfrak{m}|T_0}\sum_{j=1}^{m}(1-L_1\delta_\varepsilon)^{-j}\left(1+ \frac{e^{3L_1 + |\mathfrak{m}|T_0}L_2\delta_\varepsilon}{1-(L_1+e^{3L_1}L_2)\delta_\eps} \right)^j\delta_\varepsilon\\
              &\leq L_2e^{6L_1+|\mathfrak{m}| T_0}(1-L_1\delta_\varepsilon)^{-m}\left(1+ \frac{e^{3L_1 + |\mathfrak{m}|T_0}L_2\delta_\varepsilon}{1-(L_1+e^{3L_1}L_2)\delta_\eps} \right)^m\delta_\varepsilon m.
       \end{split}
   \end{equation}
   Now, recall from Section \ref{RelTimeScale} that $M_\varepsilon \leq 3\delta_\varepsilon^{-1}$. This implies that for small enough $\varepsilon$, we have the bound
   \begin{equation}
       \begin{split}
           A_{M_\varepsilon}(\varepsilon)&\leq 3 e^{6L_1+|\mathfrak{m}|T_0}L_2(1-L_1\delta_\varepsilon)^{-3\delta_\varepsilon^{-1}}\left(1+ \frac{e^{3L_1 + |\mathfrak{m}|T_0}L_2\delta_\varepsilon}{1-(L_1+e^{3L_1}L_2)\delta_\eps} \right)^{3\delta_{\varepsilon}^{-1}}\\
           &\leq C_{T_0,\mathfrak{m}}e^{18L_1}L_2\exp\left\{\frac{3 e^{3L_1+|\mathfrak{m}|T_0}L_2}{1-(L_1+e^{3L_1}L_2)\delta_\eps}\right\},
       \end{split}
   \end{equation}
   where $C_{T_0,\mathfrak{m}}$ is a constant depending only on $T_0$ and $\mathfrak{m}$. Substituting this bound into (\ref{lipItrBdResult}) gives
   \begin{equation}
       |D_zw_\eps(t_{M_\eps},x)|_{L^2(\mathbb{P})}\leq \frac{C_{T_0,\mathfrak{m}}e^{18L_1}L_2\exp\left\{6 e^{3L_1+|\mathfrak{m}|T_0}L_2\right\}}{(\log\varepsilon^{-1})^{1/4}}e^{\mathfrak{m}t_{M_\eps}}G_{t_{M_\eps}+\eps^2}(x-z).
   \end{equation}
   In Section~\ref{RelTimeScale}, we also defined
   \begin{equation}
    M_\varepsilon := \left\lfloor \left(2+\frac{\log(T+\varepsilon^2)}{\log\varepsilon^{-1}}\right)\delta_\varepsilon^{-1}\right\rfloor,
   \end{equation}
    so to obtain a bound at time $T$ and complete the proof, we need to iterate one last time over the interval 
    \begin{equation}
\label{finalInt}
\left[\varepsilon^{2-M_\varepsilon\delta_\varepsilon}-\varepsilon^2, T\right].
    \end{equation}
    The only difference is that we now need to account for the fact that (\ref{finalInt}) is shorter than the previous intervals. We will write
    \begin{equation}
    \label{deltaF}
    \delta^{\mathrm{f}}_\varepsilon \coloneqq 2+\frac{\log(T+\varepsilon^2)}{\log\varepsilon^{-1}}-\left\lfloor \left(2+\frac{\log(T+\varepsilon^2)}{\log\varepsilon^{-1}}\right)\delta_\varepsilon^{-1}\right\rfloor\delta_\varepsilon,
    \end{equation}
    so that
    \begin{equation}
      \log\left(\frac{T}{\varepsilon^{2-M_\varepsilon \delta_\varepsilon}}\right)  = \delta_\varepsilon^{\mathrm{f}} \log\varepsilon^{-1}\leq \delta_\varepsilon\log\varepsilon^{-1}.
    \end{equation}
    We follow the proof of Proposition \ref{lipMainEstimate}, this time skipping (\ref{rewriteStep}) and simply using (\ref{I2ineq2}) instead, and we obtain 
\begin{equation}
          \begin{split}
            |D_zw_\varepsilon(t,x)|_{L^2(\mathbb{P})}&\leq \left(\frac{C_{T_0,\mathfrak{m}}e^{18L_1}L_2\exp\left\{6 e^{3L_1+|\mathfrak{m}|T_0}L_2\right\}}{(\log\varepsilon^{-1})^{1/4}} + \frac{  L_2 e^{3L_1+|\mathfrak{m}|T_0}}{(\log\varepsilon^{-1})^{1/2}}\sqrt{\delta^{\mathrm{f}}_\varepsilon} \right)e^{\mathfrak{m}t}G_{t+\varepsilon^2}(x-z)\\
        &\quad +\frac{L_1}{\log\varepsilon^{-1}}\int_{t_{M_\varepsilon}}^t\frac{1}{s+\varepsilon^2}\int G^\mathfrak{m}_{t-s}(x-y)|D_zw_\varepsilon(s,y)|_{L^2(\mathbb{P})}\,dy\,ds\\
         &\quad + \frac{ e^{3L_1 }L_2}{\log\varepsilon^{-1}}e^{\mathfrak{m}t}G_{t+\varepsilon^2}(x-z)\int_{t_{M_\varepsilon}}^t \frac{|w_\varepsilon(s,x)|_{L^2(\mathbb{P})}}{\sqrt{s+\varepsilon^2}}\,ds
          \end{split}
      \end{equation}
      on the interval $[t_{M_\varepsilon},T]$. Because the conditions (\ref{eq:delta-eps-condition-1}) and (\ref{eq:delta-eps-condition-2}) still hold when $\delta_\eps$ is replaced with $\delta^{\mathrm f}$, the rest of the proof is essentially the same as the previous iterations. We first apply (\ref{Gronwall1}) as in Corollary~\ref{iterativeMallBd} and obtain
      \begin{equation}
          \begin{split}
            |D_zw_\varepsilon(t,x)|_{L^2(\mathbb{P})}&\leq \frac{e^{\mathfrak{m}t}}{1-L_1\delta^{\mathrm f}_\varepsilon}G_{t+\varepsilon^2}(x-z)\\& \hspace{5em} \cdot\left(\frac{C_{T_0,\mathfrak{m}}e^{18L_1}L_2\exp\left\{\frac{3 e^{3L_1+|\mathfrak{m}|T_0}L_2}{1-(L_1+e^{3L_1}L_2)\delta_\eps}\right\}}{(\log\varepsilon^{-1})^{1/4}} + \frac{ L_2 e^{3L_1+|\mathfrak{m}|T_0}}{(\log\varepsilon^{-1})^{1/2}}\sqrt{\delta^{\mathrm f}_\varepsilon} \right)
          \\&\qquad  + \frac{ e^{3L_1 }L_2}{\log\varepsilon^{-1}}e^{\mathfrak{m}t}G_{t+\varepsilon^2}(x-z)\int_{t_{M_\varepsilon}}^t \frac{|w_\varepsilon(s,x)|_{L^2(\mathbb{P})}}{\sqrt{s+\varepsilon^2}}\,ds.
          \end{split}
      \end{equation}
      Then following the proof of (\ref{itrLipSolBd}) gives
      \begin{multline}
|w_\varepsilon(t,x)|_{L^2(\mathbb{P})}\leq
\frac{e^{|\mathfrak{m}|T_0}}{(1-(L_1+e^{3L_1}L_2)\delta^{\mathrm f}_\varepsilon)\sqrt{t+\varepsilon^2}}
\\\cdot\left(\frac{C_{T_0,\mathfrak{m}}e^{18L_1}L_2\exp\left\{\frac{3 e^{3L_1+|\mathfrak{m}|T_0}L_2}{1-(L_1+e^{3L_1}L_2)\delta_\eps}\right\}}{(\log\varepsilon^{-1})^{1/4}} + \frac{ L_2 e^{3L_1+|\mathfrak{m}|T_0}}{(\log\varepsilon^{-1})^{1/2}}\sqrt{\delta^{\mathrm f}_\varepsilon}\right),
      \end{multline}
     so
      \begin{equation}
          |w_\varepsilon(T,X)|_{L^2(\mathbb{P})}\leq C_{T_0,\mathfrak{m}}'e^{18L_1}L_2\exp\left\{\frac{3 e^{3L_1+|\mathfrak{m}|T_0}L_2}{1-(L_1+e^{3L_1}L_2)\delta_\eps}\right\}\frac{(\log\varepsilon^{-1})^{-1/4}}{\sqrt{T+\varepsilon^2}},
      \end{equation}
    where $C'_{T_0,\mathfrak m }$ is a new constant, still depending only on $T_0$ and $\mathfrak m$.
     Our choice of $T\in (0,T_0]$ was arbitrary, so this completes the proof of Proposition \ref{lipSoltnBd}.

       \section{Proof of Theorem \ref{thm:main-theorem-McK-V}}
       \label{mainResultSection}
    In this section, we prove our main theorem (Theorem~\ref{thm:main-theorem-McK-V}) for general nonlinearities $f\in \bar{S}'$. (See Definition \ref{def:S}.) Our technique will be to approximate the nonlinearity $f\in \bar{S}'$ by nonlinearities in $\bar{S}_1$, allowing us to apply Proposition \ref{lipSoltnBd}. Such an approximation is possible because of the contracting properties of nonlinearities in $\bar{S}_2'$, and in particular, our concentration result proved in Lemma~\ref{concentrationIneq}. We will add a cutoff to $f'$ at a scale which increases as we take $\eps \to 0$, and use the tail estimates on $u$ to show that our approximation does not generate too much error. Similar to our proof of Proposition \ref{lipSoltnBd}, it is difficult to bound the error all at once, so we will again construct an iterative scheme for bounding the error on the short intervals $[t_m,t_{m+1}]$. (See (\ref{eq:timescale}).) Throughout this section, we will assume that $f\in \bar{S}'$, and more specifically that
    \begin{equation}
      f=f_1 + f_2,\qquad\text{with}\qquad f_1\in \bar{S}_1(L_1,L_2)\quad\text{and}\quad f_2\in \bar{S}_2'(\gamma_1,\gamma_2,\ell_1,\ell_2)
    \end{equation}
    for some $L_1,L_2,\ell_1,\ell_2\geq 0$, $\gamma_1\in [2,3)$, and $\gamma_2\in [1,2)$.
    
    \subsection{Approximating functions}
    \label{approximatingFnSection}
    Here, we construct the approximating functions used to prove Theorem \ref{thm:main-theorem-McK-V}.  The main idea is to add a cutoff to $f'$ at increasingly large scales as $\varepsilon$ tends to $0$. The assumption on the growth of $f'$ and $f''$ for $f\in \bar{S}_2'$ then provides us with the bounds required for the approximation to be in $\bar{S}_1$. By Lemma \ref{concentrationIneq}, the probability that $f(t,u_\varepsilon)$ is different from the cut-off version decreases rapidly as we take $\varepsilon$ to $0$. We will use this fact in Section~\ref{approxErrorSection} to bound the amount of error generated by our approximation. The final step will be to show it is possible to make the cutoff grow sufficiently slowly so that Proposition \ref{lipSoltnBd} still applies. We will choose our cutoff parameter $g(\varepsilon)$ in Section~\ref{convergenceSection}, but for now we impose the asymptotic behavior
    \begin{equation}
    \label{gAsymp}
       \lim_{\varepsilon \to 0} g(\varepsilon) = \infty \quad\text{and} \quad \lim_{\varepsilon \to 0} e^{3g(\varepsilon)^{\gamma_1}}g(\varepsilon)^{\gamma_2}\delta_\varepsilon = 0.
    \end{equation}
    This means that $g(\eps)$ is diverging as $\eps\to 0$, but at an extremely slow (at most a fractional power of logarithmic, and eventually we will choose it to be iterated logarithmic) rate.
       
    \begin{lemma}
    \label{approxNonlinearity}
    There exists a family of functions $\{\tilde{f}_\eps\}_\eps$ of the form
        \begin{equation}
            \tilde{f}_\eps = f_1 + \tilde{f}_{2,\eps},
        \end{equation}
        such that the following hold:
        \begin{enumerate}[(1)]
        \label{enum:approx-nonlinearity-properties}
        \item We have \begin{equation}\tilde{f}_\eps(t,u) = f(t,u)\qquad\text{whenever}\qquad |u|\leq \frac{g(\eps)}{\sqrt{t}}.\label{eq:usually-agree}\end{equation}
            \item We have $\tilde f_{2,\eps} \in \bar{S}_2'(\gamma_1,\gamma_2,\ell_1,\ell_2)$.
            \item We have  \begin{equation}\tilde{f}_\eps \in \bar{S}_1(\tilde{L}_1(\eps),\tilde{L}_2(\eps)),\label{eq:ftildeinS1}
              \end{equation}
              with $\tilde{L}_1(\varepsilon) \coloneqq L_1+\ell_1(1+g(\varepsilon)^{\gamma_1})$ and $\tilde{L}_2(\varepsilon) \coloneqq L_2+\ell_2(1+g(\varepsilon)^{\gamma_2})$.
        \end{enumerate}
    \end{lemma}
    \begin{proof}
For $t\ge 0$ and $u\in \mathbb{R}$, we define
        \begin{equation}\tilde f_{2,\eps}(t,u) = \int_0^u f_2'(t,|p|\wedge (g(\eps)t^{-1/2}))\,dp.
        \end{equation} In particular, when $|u|\le g(\eps)t^{-1/2}$, we have $\tilde f_{2,\eps}(t,u) = \int_0^u f_2'(t,|p|)\,dp = \int_0^u f_2'(t,p)\,dp = f_2(t,u)$ since $f_2(t,\cdot)$ is odd and so $f_2'(t,\cdot)$ is even. Thus if we set $\tilde f_\eps = f_1 + \tilde f_{2,\eps}$, then (\ref{eq:usually-agree}) holds. The fact that $\tilde{f}_\eps\in \bar{S}_1(\tilde{L}_1(\eps),\tilde{L}_2(\eps))$
       follows from our assumption that $f_2\in \bar{S}_2'(\gamma_1,\gamma_2,\ell_1,\ell_2)$, which gives the bounds
        \begin{equation}
            |f_2'(t,u)|\leq \frac{\ell_1(1+g(\varepsilon)^{\gamma_1})}{t}\quad\text{and}\quad  \Lip(f_2'(t,\cdot)|_{[0,u]})\leq \frac{\ell_2(1+g(\varepsilon)^{\gamma_2})}{\sqrt{t}}
        \end{equation}
        for all $|u|\leq \frac{g(\varepsilon)}{\sqrt{t}}$.
    \end{proof}
    We define the approximating solutions that will be used to prove Theorem \ref{thm:main-theorem-McK-V} as follows.
    \begin{definition}
        We let $\tilde{u}_\varepsilon$ and $\tilde{v}_\varepsilon$ be the solutions to
        \begin{equation}
                \begin{cases}
                    \partial_t \tilde{u}_\varepsilon = \frac{1}{2}\Delta \tilde{u}_\varepsilon + \mathfrak{m}\tilde{u}_\varepsilon - \frac{1}{\log\varepsilon^{-1}}\tilde{f}_\eps(t+\eps^2,\tilde{u}_\varepsilon)\\
                    \tilde{u}_\varepsilon(0,x) = u_\varepsilon(0,x)
                  \end{cases}\label{eq:cutoffu}
        \end{equation}
        and
        \begin{equation}
        \label{cutoffMckeanVlasov}
               \begin{cases}
                    \partial_t \tilde{v}_\varepsilon = \frac{1}{2}\Delta \tilde{v}_\varepsilon + \mathfrak{m}\tilde{v}_\varepsilon - \frac{1}{\log\varepsilon^{-1}}\mathbb{E}[\tilde{f}_\eps'(t+\eps^2,\tilde{v}_\varepsilon)]\tilde{v}_\varepsilon\\
                    \tilde{v}_\varepsilon(0,x) = v_\varepsilon(0,x)
               \end{cases}
        \end{equation}
        respectively, where $\tilde{f}_{\varepsilon}$ is given by Lemma \ref{approxNonlinearity}.
    \end{definition}
      Since $\tilde f_\eps\in \bar S'$, our estimates in Section~\ref{prelimEstimates} still hold, and in particular the McKean--Vlasov equation (\ref{cutoffMckeanVlasov}) is well-posed by the results of Section~\ref{ExistenceUniquenessSection}.
    
         \subsection{Sub-/super-solutions for the error}
    \label{subSupSection}
    We define the difference
    \begin{equation}
        U_\eps := \tilde{u}_\eps - u_\eps.
    \end{equation}
    To motivate the use of sub/super-solutions to control the error generated by our approximations, let us first rewrite the PDE for the difference $D_zU_{\eps}= D_z\tilde{u}_\varepsilon-D_z u_\varepsilon$. We see that  
    \begin{equation}
    \label{approxPDE}
        \begin{split}
          \partial_t D_zU_\varepsilon & = \frac{1}{2}\Delta D_z U_\varepsilon + \mathfrak{m} D_zU_\varepsilon - \frac{\tilde{f}_{\varepsilon}'(t+\eps^2,\tilde{u}_\varepsilon)}{\log\varepsilon^{-1}}D_z\tilde{u}_\varepsilon + \frac{f'(t+\eps^2,u_\varepsilon)}{\log\varepsilon^{-1}}D_zu_\varepsilon\\
              &\ = \frac{1}{2}\Delta D_z U_\varepsilon + \mathfrak{m} D_zU_\varepsilon- \frac{f'_1(t+\eps^2,u_\varepsilon)}{\log\varepsilon^{-1}}D_zU_\varepsilon-\frac{f'_2(t+\eps^2,u_\varepsilon)}{\log\varepsilon^{-1}}D_zU_\varepsilon 
            +E_{1;\eps}(t,x)+E_{2;\eps}(t,x),
        \end{split}
    \end{equation}
   where we have defined
        \begin{equation}
            E_{1;\eps}(t,x) \coloneqq - \frac{f_1'(t+\eps^2,\tilde{u}_\varepsilon(t,x))-f_1'(t+\eps^2,u_\varepsilon(t,x))}{\log\varepsilon^{-1}}D_z\tilde{u}_\varepsilon(t,x)
          \end{equation}
          and
          \begin{equation}
            E_{2;\eps}(t,x)\coloneqq -\frac{\tilde{f}_{2,\varepsilon}'(t+\eps^2,\tilde{u}_\varepsilon(t,x))-f_2'(t+\eps^2,u_\varepsilon(t,x))}{\log\varepsilon^{-1}}D_z\tilde{u}_\varepsilon(t,x) .
      \end{equation}
    Because $f_2'$ can be large, the term $-\frac{f_2'(t+\eps^2,u_\eps)}{\log\varepsilon^{-1}}D_zU_\varepsilon$ in (\ref{approxPDE}) is potentially dangerous. However, this term should actually be helping us by the assumption that $f_2'\geq 0$. Still, it is difficult to use this fact because we do not know that $D_zU_\varepsilon$ is signed. Instead, we consider the worst case for the other terms in the nonlinearity, which will provide us with sub-/super-solutions of $D_zU_\varepsilon$ that have signs. This will allow us to ignore the problematic term $ - \frac{f_2'(t+\eps^2,u_\eps)}{\log\eps^{-1}}D_zU_\eps$ similarly to our construction in Lemma~\ref{MckeanUnique}. Let 
    \begin{equation}
        E_\eps(t,x):=  E_{1;\eps}(t,x) +  E_{2;\eps}(t,x).
    \end{equation}
    \begin{lemma}
    \label{approxSubSup}
        We have
        \begin{equation}
        \label{approxSubSupIneq}
            \underline{D_zU_\varepsilon}\leq D_zU_\varepsilon \leq \overline{D_zU_\varepsilon},
        \end{equation}
        where $\underline{D_zU_\varepsilon}$ and $\overline{D_zU_\varepsilon}$ solve the equations
        \begin{equation}
        \label{eq:sub-sol-pde}
           \begin{split}
                \partial_t \underline{D_zU_\varepsilon} &= \frac{1}{2}\Delta  \underline{D_zU_\varepsilon}  + \left(\mathfrak{m} + \frac{L_1}{(t+\varepsilon^2)\log\varepsilon^{-1}}\right) \underline{D_zU_\varepsilon} - |E_\eps(t,x)|
           \end{split}
        \end{equation}
        and 
         \begin{equation}
        \label{eq:sup-sol-pde}
                \partial_t \overline{D_zU_\varepsilon} = \frac{1}{2}\Delta  \overline{D_zU_\varepsilon}  + \left(\mathfrak{m} + \frac{L_1}{(t+\varepsilon^2)\log\varepsilon^{-1}}\right) \overline{D_zU_\varepsilon} + |E_\eps(t,x)|,
        \end{equation}
        respectively, with the initial conditions
        \begin{equation}
          \underline{D_zU_\varepsilon}(0,x) = 0 =  \overline{D_zU_\varepsilon}(0,x).\label{eq:sub-sup-ic}
        \end{equation}
    \end{lemma}
  The reader should note that the symbols $\underline{D_zU_\eps}$ and $\overline{D_zU_\eps}$ are to be interpreted as atomic symbols for objects defined via (\ref{eq:sub-sol-pde}--\ref{eq:sub-sup-ic}), not as the result of any particular mathematical operations applied to $D_zU_\eps$. We hope that the mnemonic utility of these symbols is worth this slight abuse of notation.
    \begin{proof}
      We first note that by the maximum principle (or simply the Duhamel formulas for the linear equations~(\ref{eq:sub-sol-pde}) and~(\ref{eq:sup-sol-pde})), we have 
      \begin{equation}\label{eq:Dzsubsupsigns}
           \underline{D_zU_\eps}(t,x)\le 0\le\overline{D_zU_\eps}(t,x)
        \end{equation}
        for all $t\in [0,T]$ and $x\in \mathbb{R}^2$. Now we define the parabolic operator
        \begin{equation}
          \mathcal{L}U := \frac12\Delta U  + \mathfrak{m}U -\frac{1}{\log\eps^{-1}}f'(t+\eps^2,u_\eps)U - \partial_t U
        \end{equation}
        and observe that
        \begin{equation}
            \begin{split}
                \mathcal{L}[\overline{D_zU_\eps}-D_zU_\eps ] &= -\frac{1}{\log\eps^{-1}}\left(f_1'(t+\eps^2,u_\eps) + \frac{L_1}{t+\eps^2}\right)\overline{D_zU_\eps}
                \\&\qquad - \frac{f_2'(t+\eps^2,u_\eps)}{\log\eps^{-1}}\overline{D_zU_\eps} +E_\eps(t,x)-|E_\eps(t,x)|\\
                &\leq 0
            \end{split}
        \end{equation}
        by (\ref{S1Bd}), (\ref{eq:S2derivpos}), and (\ref{eq:Dzsubsupsigns}).
        Hence, the maximum principle 
        implies that
        \begin{equation}
            D_zU_\eps \leq \overline{D_zU_\eps}.
        \end{equation}
        Similarly, we have $\mathcal{L}[D_zU_\eps-\underline{D_zU_\eps} ]\leq 0$, which completes the proof of~(\ref{approxSubSupIneq}).
    \end{proof}
       \subsection{Iterative error bounds for the approximation}
       \label{approxErrorSection}
       By symmetry, we have $|\overline{D_zU_\eps}|_{L^p(\mathbb{P})}=|\underline{D_zU_\eps}|_{L^p(\mathbb{P})}$. Therefore, by Lemma~\ref{approxSubSup}, we can control $|D_zU_\varepsilon|_{L^2(\mathbb{P})}$, and therefore $|U_\eps|_{L^2(\mathbb{P})}$ using the Gaussian Poincar\'e inequality, by bounding $|\overline{D_zU_\varepsilon}|_{L^2(\mathbb{{P}})}$. We will use a similar iterative scheme to the one constructed in Section~\ref{itrSection}. That is, we will alternate between bounding the error at the level of the Malliavin derivatives and at the level of the solutions on each of the short intervals $[t_m,t_{m+1}]$. (See~(\ref{eq:timescale}).) The key to this estimate is our concentration inequality (\ref{concentrationIneq}), which by our construction in (\ref{approxNonlinearity}) implies that the probability that $f$ is different than $\tilde{f}_\eps$ decreases rapidly as we take $\eps \to 0$. We will work with a fixed time horizon $T_0\in (0,\infty)$ and an arbitrary $T\in (0,T_0]$.
    
    The first step in our iteration is to bound $|\overline{D_zU_\varepsilon}|_{L^2(\mathbb{{P}})}$ on the short interval $[t_m,t_{m+1}]$ by a small error term and a term which depends on the difference of the solutions $U_\varepsilon$. 

    \begin{lemma}
    \label{supSupShortIntBd}
        Suppose that for some $A<\infty$, we have the bound
        \begin{equation}
        \label{supSolIndHyp}
        |\overline{D_zU_\varepsilon}(t_m,x)|_{L^2(\mathbb{{P}})} \leq A\exp\left\{-\frac{g(\varepsilon)^2}{8e^{2|\mathfrak{m}|T_0+6L_1}}\right\} e^{\mathfrak{m}t_m}G_{t_m+\varepsilon^2}(x - z)\qquad\text{for all }x\in\mathbb{R}^2.
        \end{equation}
        Then for all $t\in [t_m,t_{m+1}]$ and $x\in \mathbb{R}^2$, we have, recalling the constants $K_{p,T_0}$ defined in Proposition~\ref{concentrationIneq},
        \begin{equation}
        \label{subSupShortIntIneq}
            \begin{split}
                |\overline{D_zU_\varepsilon}(t,x)|_{L^2(\mathbb{P})} &\leq e^{3L_1\delta_\eps}\left(A+ 2\ell_1K_{4\gamma_1,T_0}^{1/4}\delta_\varepsilon\right)\exp\left\{-\frac{g(\varepsilon)^2}{8e^{2|\mathfrak{m}|T_0+6L_1}}\right\}e^{\mathfrak{m}t} G_{t+\varepsilon^2}(x - z)\\ 
                &\qquad+\frac{ e^{3L_1(\delta_\varepsilon + 1)}(L_2+\tilde{L}_2(\eps))}{\log\varepsilon^{-1}}e^{\mathfrak{m}t}G_{t+\varepsilon^2}(x-z)\int_{t_m}^t \frac{|U_\varepsilon(s,x)|_{L^2(\mathbb{P})}}{\sqrt{s+\varepsilon^2}}\,ds.
            \end{split}
        \end{equation}
    \end{lemma}
    \begin{proof}
        Writing Duhamel's formula for the solution to (\ref{eq:sup-sol-pde}) and taking $|\cdot|_{L^2(\mathbb{P})}$ of both sides gives, by the triangle inequality, that
        \begin{equation}
            \begin{split}
              |&\overline{D_zU_\varepsilon}(t,x)|_{L^2(\mathbb{P})}\\&\leq e^{3L_1\delta_\varepsilon}\int G^\mathfrak{m}_{t-t_m}(x-y)|\overline{D_zU_\varepsilon}(t_m,y)|_{L^2(\mathbb{P})}\,dy\\
                & \qquad+ \frac{e^{3L_1\delta_\varepsilon}L_2 }{\log\varepsilon^{-1}}\int_{t_m}^t\int G^\mathfrak{m}_{t-s}(x-y)\frac{|U_\varepsilon(s,x)D_zu_\varepsilon(s,y)|_{L^2(\mathbb{P})}}{\sqrt{s+\varepsilon^2}}\,dy\,ds\\
                &\qquad + \frac{e^{3L_1\delta_\varepsilon}}{\log\varepsilon^{-1}}\int_{t_m}^t\int G^\mathfrak{m}_{t-s}(x-y)
              \\&\hspace{10em}\cdot\left|\Bigl{(}\tilde{f}_{2,\varepsilon}'(s+\eps^2,\tilde{u}_\varepsilon(s,y))-\tilde{f}_{2,\varepsilon}'(s+\eps^2,u_\varepsilon(s,y))\Bigr{)}D_z\tilde{u}_\varepsilon(s,y)\right|_{L^2(\mathbb{P})}\,dy\,ds\\
                &\qquad+ \frac{e^{3L_1\delta_\varepsilon}}{\log\varepsilon^{-1}}\int_{t_m}^t\int G^\mathfrak{m}_{t-s}(x-y)
                \\&\hspace{10em}\cdot\left|
                \Big{(}\tilde{f}_{2,\varepsilon}'(s+\eps^2,u_\varepsilon(s,y))-f_2'(s+\eps^2,u_\varepsilon(s,y))\Big{)}D_z\tilde{u}_\varepsilon(s,y)\right|_{L^2(\mathbb{P})}\,dy\,ds,
            \end{split}
        \end{equation}
        where we used (\ref{eq:reasonfortms}), and also the assumption that $f_1\in \bar{S}_1(L_1,L_2)$ to bound the term $|E_{1;\eps}(s,y)|$ by $L_2(\log\eps^{-1}\sqrt{s+\eps^2})^{-1}|U_\eps(s,x)D_zu_\eps(s,y)|$. Then by Lemma~\ref{ptwMallbd}, our assumption (\ref{supSolIndHyp}), and Chapman--Kolmogorov, we have
         \begin{equation}
         \label{supSolDuhamel}
            \begin{split}
                |\overline{D_zU_\varepsilon}(t,x)|_{L^2(\mathbb{P})}&\leq Ae^{3L_1\delta_\varepsilon}\exp\left\{-\frac{g(\varepsilon)^2}{8e^{2|\mathfrak{m}|T_0+6L_1}}\right\} e^{\mathfrak{m}t}G_{t+\varepsilon^2}(x - z)\\
                &\qquad+ \frac{ e^{3L_1(\delta_\varepsilon+1)}  L_2}{\log\varepsilon^{-1}}e^{\mathfrak{m}t}G_{t+\varepsilon^2}(x-z)\int_{t_m}^t\frac{|U_\varepsilon(s,x)|_{L^2(\mathbb{P})}}{\sqrt{s+\varepsilon^2}}\,ds \\
                &\qquad+ \frac{ e^{3L_1(\delta_\varepsilon+1)} }{\log\varepsilon^{-1}}e^{\mathfrak{m}t}G_{t+\varepsilon^2}(x-z)\\
                &\hspace{5em}\cdot\int_{t_m}^t|\tilde{f}_{2,\varepsilon}'(s+\eps^2,\tilde{u}_\varepsilon(s,x))-\tilde{f}_{2,\varepsilon}'(s+\eps^2,u_\varepsilon(s,x))|_{L^2(\mathbb{P})}\,ds\\
                &\qquad+ \frac{ e^{3L_1(\delta_\varepsilon+1)} }{\log\varepsilon^{-1}}e^{\mathfrak{m}t}G_{t+\varepsilon^2}(x-z)
                \\&\hspace{5em}\cdot\int_{t_m}^t |
                \tilde{f}_{2,\varepsilon}'(s+\eps^2,u_\varepsilon(s,x))-f_2'(s+\eps^2,u_\varepsilon(s,x))|_{L^2(\mathbb{P})}\,ds.
            \end{split}
        \end{equation}
        By~(\ref{eq:ftildeinS1}), we have  
        \begin{equation}
        \label{contractingItrBd1}
        |\tilde{f}_{2,\varepsilon}'(s+\eps^2,\tilde{u}_\varepsilon(s,x))-\tilde{f}_{2,\varepsilon}'(s+\eps^2,u_\varepsilon(s,x))|_{L^2(\mathbb{P})} \leq \frac{\tilde{L}_2(\eps)}{\sqrt{s+\varepsilon^2}}|U_\varepsilon(s,x)|_{L^2(\mathbb{P})}.
        \end{equation}
        Moreover, we can write using (\ref{eq:usually-agree}) that
        \begin{equation}
        \begin{split}
        \label{contractingItrBd2}
        \left|\tilde{f}_{2,\varepsilon}'(s+\eps^2,u_\varepsilon(s,x))-f_2'(s+\eps^2,u_\varepsilon(s,x))\right|_{L^2(\mathbb{P})}\\&\hspace{-15em}=  \left|(
        \tilde{f}_{2,\varepsilon}'(s+\eps^2,u_\varepsilon(s,x))-f_2'(s+\eps^2,u_\varepsilon(s,x)))\mathbf{1}\left\{|u_\varepsilon(s,x)|\geq \frac{g(\varepsilon)}{\sqrt{s+\varepsilon^2}}\right\}\right|_{L^2(\mathbb{P})}\\
              &\hspace{-15em}\leq 2\ell_1\bigl||u_\varepsilon(s,x)|^{\gamma_1}\bigr|_{L^4(\mathbb{P})}(s+\varepsilon^2)^{\frac{\gamma_1-2}{2}}  \mathbb{P}\left(|u_\varepsilon(s,x)|\geq \frac{g(\varepsilon)}{\sqrt{s+\varepsilon^2}}\right)^{1/4}\\
                &\hspace{-15em}\leq \frac{ 2\ell_1K_{4\gamma_1,T_0}^{1/4}}{s+\varepsilon^2}\exp\left\{-\frac{g(\varepsilon)^2}{8e^{2|\mathfrak{m}|T_0+6L_1}}\right\},
        \end{split}
        \end{equation}
        where in the first inequality we used the Cauchy--Schwarz inequality as well as (\ref{eq:ftildeinS1}),
        and in the last line
        we used both inequalities in Lemma~\ref{concentrationIneq}. Using~(\ref{contractingItrBd1}) and~(\ref{contractingItrBd2}) in~(\ref{supSolDuhamel}) gives
        \begin{equation}
            \begin{split}
                |\overline{D_zU_\varepsilon}(t,x)|_{L^2(\mathbb{P})} &\leq e^{3L_1\delta_\varepsilon}\left(A + 2\ell_1K_{4\gamma_1,T}^{1/4}\delta_\varepsilon\right)\exp\left\{-\frac{g(\varepsilon)^2}{8e^{2|\mathfrak{m}|T+6L_1}}\right\}e^{\mathfrak{mt}}G_{t+\varepsilon^2}(x - z)\\ 
                &\qquad+\frac{ e^{3L_1(\delta_\varepsilon + 1)}(L_2+\tilde{L}_2(\eps))}{\log\varepsilon^{-1}}e^{\mathfrak{m}t}G_{t+\varepsilon^2}(x-z)\int_{t_m}^t \frac{|U_\varepsilon(s,x)|_{L^2(\mathbb{P})}}{\sqrt{s+\varepsilon^2}}\,ds,
            \end{split}
        \end{equation}
        which is (\ref{subSupShortIntIneq}).
    \end{proof}

     The next step is to obtain a bound on $|U_\varepsilon|_{L^2(\mathbb{P})}$ over the interval $[t_m,t_{m+1}]$ using Lemma \ref{supSupShortIntBd} and the Gaussian Poincar\'e inequality. 
    \begin{corollary}\label{contractingItrSolBd}
      Suppose that (\ref{supSolIndHyp}) holds.
        Then for sufficiently small $\eps$ depending only on $\mathfrak{m}$, $T_0$, $L_1$, $L_2$, $\ell_2$, $\gamma_2$, and our choice of the function $g$, we have for all $t\in [t_m,t_{m+1}]$ and $x\in \mathbb{R}^2$ that  
        \begin{multline}
        \label{eq:contractingItrSolBd}
|U_\varepsilon(t,x)|_{L^2(\mathbb{P})}\leq \frac{e^{3L_1\delta_\varepsilon+\mathfrak{m}t}}{\sqrt{t+\eps^2}}\left(1- e^{3L_1(\delta_\varepsilon + 1)+|\mathfrak{m}|T_0}(L_2+\tilde{L}_2(\eps))\delta_\varepsilon\right)^{-1}
            \\\cdot\left(A +2\ell_1K_{4\gamma_1,T}^{1/4}\delta_\varepsilon \right)\exp\left\{-\frac{g(\varepsilon)^2}{8e^{2|\mathfrak{m}|T_0+6L_1}}\right\}.
        \end{multline}
    \end{corollary}
    \begin{proof}
    To simplify notation, we set
    \begin{equation}
        \kappa(\varepsilon)\coloneqq e^{3L_1\delta_\varepsilon}\left(A+2\ell_1K_{4\gamma_1,T_0}^{1/4}\delta_\varepsilon\right)\exp\left\{-\frac{g(\varepsilon)^2}{8e^{2|\mathfrak{m}|T_0+6L_1}}\right\}.
    \end{equation}
        By the Gaussian Poincaré inequality and the symmetry of (\ref{approxSubSupIneq}), we have
        \begin{equation}
            \begin{split}
              |U_\eps(t,x)|_{L^2(\mathbb{P})}  &\leq \left(\int \mathbb{E}[D_zU_\varepsilon(t,x)^2]\,dz\right)^{1/2}\\
                                               &\leq \left(\int \mathbb{E}[\overline{D_zU_\varepsilon}(t,x)^2]\,dz\right)^{1/2}\\
                                               &\leq \frac{e^{\mathfrak{m}t}}{\sqrt{t+\varepsilon^2}} \left(\kappa(\varepsilon)+\frac{ e^{3L_1(\delta_\varepsilon + 1)}(L_2+\tilde{L}_2(\eps))}{\log\varepsilon^{-1}}\int_{t_m}^t \frac{|U_\varepsilon(s,x)|_{L^2(\mathbb{P})}}{\sqrt{s+\varepsilon^2}}\,ds\right),
            \end{split}
        \end{equation}
        where the last line follows from Lemma~\ref{supSupShortIntBd}. Then, using our assumption (\ref{gAsymp}),  we may apply  Lemma~\ref{Gronwall2} with
        \begin{equation}
            \alpha = e^{\mathfrak{m}t}\kappa(\varepsilon)
        \end{equation}
        and
        \begin{equation}
            \beta =  e^{3L_1(\delta_\varepsilon + 1)+|\mathfrak{m}|T_0}(L_2+\tilde{L}_2(\eps)),
        \end{equation}
       which gives (\ref{eq:contractingItrSolBd}).
           \end{proof}

       The last step is to substitute the bound from Corollary \ref{contractingItrSolBd} back into the bound we obtained in Lemma \ref{supSupShortIntBd} in order to move on to the next interval.
    \begin{corollary}
    \label{contractingInductionStep}
         Suppose that for some $A_m\geq0$, we have the bound
        \begin{equation}
            |\overline{D_zU_\varepsilon}(t_m,\cdot)|_{L^2(\mathbb{{P}})} \leq A_m\exp\left\{-\frac{g(\varepsilon)^2}{8e^{2|\mathfrak{m}|T_0+6L_1}}\right\}e^{\mathfrak{m}t_{m}}G_{t_m+\varepsilon^2}(\cdot - z);
            \end{equation}
      i.e.\ (\ref{supSolIndHyp}) holds with $A=A_m$.
            Then, for sufficiently small $\eps$ depending only on $\mathfrak{m}$, $T_0$, $L_1$, $L_2$, $\ell_2$, $\gamma_2$, and our choice of the function $g$, we have 
        \begin{equation}
        \label{eq:contractingInductionStep-conclusion}
            |\overline{D_zU_\varepsilon}(t_{m+1},\cdot)|_{L^2(\mathbb{{P}})} \leq A_{m+1}\exp\left\{-\frac{g(\varepsilon)^2}{8e^{2|\mathfrak{m}|T_0+6L_1}}\right\} e^{\mathfrak{m}t_{m+1}}G_{t_{m+1}+\varepsilon^2}(\cdot - z),
            \end{equation}
            where 
            \begin{equation}
               A_{m+1} = e^{3L_1\delta_\varepsilon}\left(A_m + 2\ell_1K_{4\gamma_1,T_0}^{1/4}\delta_\varepsilon\right)\left(1+4 \ell_2g(\varepsilon)^{\gamma_2}e^{9L_1+2|\mathfrak{m}|T}\delta_\varepsilon\right).
            \end{equation}
    \end{corollary}
    \begin{proof}
        To simplify later notation, we first take $\varepsilon$ sufficiently small so that 
        \begin{equation}
          e^{3L_1(\delta_\varepsilon + 1)+|\mathfrak{m}|T_0}(L_2+\tilde{L}_2(\eps))\leq 2 \ell_2g(\varepsilon)^{\gamma_2}e^{6L_1+|\mathfrak{m}|T_0},\label{eq:simplifying-assumption}
        \end{equation}
        which we can see is possible by examining the definition of $\tilde L_2(\eps)$ in the statement of Lemma~\ref{approxNonlinearity} and recalling (\ref{gAsymp}).
        Using this assumption in the conclusion of Corollary~\ref{contractingItrSolBd}, we obtain
          \begin{multline}
            |U_\varepsilon(t,x)|_{L^2(\mathbb{P})}\\\leq e^{3L_1\delta_\varepsilon+\mathfrak{m}t}\left(1- 2 \ell_2g(\varepsilon)^{\gamma_2}e^{6L_1+|\mathfrak{m}|T_0}\delta_\varepsilon\right)^{-1}\left(A_m + 2\ell_1K_{4\gamma_1,T_0}^{1/4}\delta_\varepsilon\right)\exp\left\{-\frac{g(\varepsilon)^2}{8e^{2|\mathfrak{m}|T_0+6L_1}}\right\}.\label{eq:Ubd}
        \end{multline}
        To further simplify later notation, we also use (\ref{gAsymp}) to take $\varepsilon$ sufficiently small so that the right side of (\ref{eq:Ubd}) is bounded by
        \begin{equation*}
              2e^{3L_1+|
              \mathfrak{m}|T_0}\left(A_m + 2\ell_1K_{4\gamma_1,T_0}^{1/4}\delta_\varepsilon\right)\exp\left\{-\frac{g(\varepsilon)^2}{8e^{2|\mathfrak{m}|T_0+6L_1}}\right\}.
        \end{equation*}
        Using this as well as (\ref{eq:simplifying-assumption}) again in (\ref{subSupShortIntIneq}), we get
         \begin{equation}
            \begin{split}
                |\overline{D_zU_\varepsilon}(t,x)|_{L^2(\mathbb{P})} &\leq e^{3L_1\delta_\varepsilon}\left(A_m+ 2\ell_1K_{4\gamma_1,T_0}^{1/4}\delta_\varepsilon\right)\exp\left\{-\frac{g(\varepsilon)^2}{8e^{2|\mathfrak{m}|T_0+6L_1}}\right\} e^{\mathfrak{m}t}G_{t+\varepsilon^2}(x - z)\\ 
                &\qquad+\frac{2 \ell_2g(\varepsilon)e^{6L_1+|\mathfrak{m}|T_0}}{\log\varepsilon^{-1}}e^{\mathfrak{m}t}G_{t+\varepsilon^2}(x-z)\int_{t_m}^t \frac{|U_\varepsilon(s,x)|_{L^2(\mathbb{P})}}{\sqrt{s+\varepsilon^2}}\,ds\\
                &\leq e^{3L_1\delta_\varepsilon}\left(A_m+2\ell_1K_{4\gamma_1,T_0}^{1/4}\delta_\varepsilon\right)\left(1+4 \ell_2g(\varepsilon)^{\gamma_2}e^{9L_1+2|\mathfrak{m}|T_0}\delta_\varepsilon\right)
                \\&\hspace{15em} \cdot\exp\left\{-\frac{g(\varepsilon)^2}{8e^{2|\mathfrak{m}|T_0+6L_1}}\right\}e^{\mathfrak{m}t}G_{t+\varepsilon^2}(x-z),
            \end{split}
        \end{equation}
      which is (\ref{eq:contractingInductionStep-conclusion}).
    \end{proof}
    We will now show that $\tilde{u}_\varepsilon$ is a good approximation for $u_\varepsilon$ by iterating our argument until we are able to bound the error at the final time $T$.
    \begin{lemma}
    \label{approxConverges}
        For all $X\in \mathbb{R}^2$, we have
        \begin{equation}
        \label{eq:approx-converges}
            \adjustlimits \lim_{\varepsilon\to 0} \sup_{T\in [0,T_0]}(T+\eps^2)^{1/2}|u_\varepsilon(T,X)-\tilde{u}_\varepsilon(T,X)|_{L^2(\mathbb{P})}=0.
        \end{equation}
    \end{lemma}
    \begin{proof}
        We see that the conditions of Corollary~\ref{contractingInductionStep} are satisfied for $m=0$ because $\overline{D_zU_\varepsilon}(0,\cdot)=0$. It follows by induction using Corollary~\ref{contractingInductionStep} that for sufficiently small $\eps$ depending only on $\mathfrak{m}$, $T_0$, $L_1$, $L_2$, $\ell_2$, $\gamma_2$, and our choice of the function $g$, the following inequality holds for all $0\leq m \leq M_\varepsilon$: 
           \begin{equation}
            \begin{split}
            \label{finalApproxErrorEst}
                |D_zU_\varepsilon(t_m,X)|_{L^2(\mathbb{P})}&\leq 
            A_{m}(\eps)\exp\left\{-\frac{g(\varepsilon)^2}{8e^{2|\mathfrak{m}|T_0+6L_1}}\right\}e^{\mathfrak{m}t_m}G_{t_m+\eps^2}(x-z),
            \end{split}
        \end{equation}
        where $A_0 = 0$ and $A_{m+1}(\varepsilon)$ is defined inductively by
        \begin{equation}
                A_{m+1}(\varepsilon) = e^{3L_1\delta_\varepsilon}\left(A_m(\varepsilon)+2\ell_1K_{4\gamma_1,T}^{1/4}\delta_\varepsilon\right)\left(1+4 \ell_2g(\varepsilon)^{\gamma_2}e^{9L_1+2|\mathfrak{m}|T_0}\delta_\varepsilon\right).
            \end{equation}
            In particular, we have
            \begin{equation}
                \begin{split}
                    A_{m+1}(\varepsilon) &\leq  2e^{3L_1m\delta_\varepsilon}\ell_1K_{4\gamma_1,T_0}^{1/4}\sum_{j=1}^m \left(1+4 \ell_2 g(\varepsilon)^{\gamma_2}e^{9L_1+2|\mathfrak{m}|T_0}\delta_\varepsilon\right)^j\delta_\varepsilon\\
                    &\leq 2e^{3L_1m\delta_\varepsilon}\ell_1K_{4\gamma_1,T_0}^{1/4}\left(1+4 \ell_2 g(\varepsilon)^{\gamma_2}e^{9L_1+2|\mathfrak{m}|T_0}\delta_\varepsilon\right)^m m\delta_\varepsilon.
                \end{split}
            \end{equation}
            Then because $M_\varepsilon \leq 3\delta_\varepsilon^{-1}$, we have 
            \begin{equation}
            \label{Eq:AmBd}
                  \begin{split}
                      A_{M_\varepsilon} &\leq  6e^{9L_1}\ell_1K_{4\gamma_1,T}^{1/4}\left(1+4 \ell_2 g(\varepsilon)^{\gamma_2}e^{9L_1+2|\mathfrak{m}|T}\delta_\varepsilon\right)^{3\delta_{\varepsilon}^{-1}}\\
                   &\leq C\exp\left\{12 \ell_2 e^{9L_1+2|\mathfrak{m}|T} g(\varepsilon)^{\gamma_2}\right\},
                  \end{split}
            \end{equation}
            where $C$ is a constant depending  on $K_{4\gamma_1,T_0}$, $\ell_1$, and $L_1$. Substituting this bound into (\ref{finalApproxErrorEst}) gives
            \begin{equation}
                |D_zU_\varepsilon(t_{M_\varepsilon},X)|_{L^2(\mathbb{P})}\leq C\exp\left\{12 \ell_2 e^{9L_1+2|\mathfrak{m}|T_0} g(\varepsilon)^{\gamma_2}-\frac{g(\varepsilon)^2}{8e^{2|\mathfrak{m}|T_0+6L_1}}\right\}e^{\mathfrak{m}t_{M_\eps}}G_{t_{M_\eps}+\eps^2}(x-z).
            \end{equation}
             To complete the proof, we need to apply one more iteration over the final interval (\ref{finalInt}) and then follow the proof of Corollary \ref{contractingItrSolBd} to obtain that for sufficiently small $\eps$ depending only on $\mathfrak{m}$, $T_0$, $L_1$, $L_2$, $\ell_2$, $\gamma_2$, and our choice of the function $g$, we have
           \begin{equation}
           \label{eq:final-bd-approx}
            (T+\eps^2)^{1/2}|U_\eps(T,X)|_{L^2(\mathbb{P})}\leq C'\exp\left\{12 \ell_2 e^{9L_1+2|\mathfrak{m}|T_0} g(\varepsilon)^{\gamma_2}-\frac{g(\varepsilon)^2}{8e^{2|\mathfrak{m}|T_0+6L_1}}\right\}.
             \end{equation}
             The argument is essentially the same as the previous iterations, where instead of $\delta_\varepsilon$, we use $\delta^{\mathrm{f}}_\varepsilon$ as in $(\ref{deltaF})$, so we omit the details. 
             The right hand side of (\ref{eq:final-bd-approx}) goes to $0$ as $\varepsilon \to 0$ because we assumed in (\ref{S2Bd}) that $\gamma_2<2$, and our choice of $T\in (0,T_0]$ was arbitrary, so this gives (\ref{eq:approx-converges}).
    \end{proof}
    
    We conclude this section by showing that $\tilde{v}_\varepsilon$ is also a good approximation for $v_\varepsilon$. Although $v_\eps$ is indeed a special case of $u_\eps$, the approximation (\ref{cutoffMckeanVlasov}) is not the corresponding special case of (\ref{eq:cutoffu}),
    so a separate proof is required. The argument will be similar to the one used to prove Lemma~\ref{approxConverges}, but the results of Section~\ref{ExistenceUniquenessSection} allow us to work at level of the initial value problem (\ref{epsilonIVP}), which simplifies the proof.
    \begin{lemma}
    \label{approxVConverges}
        For all $X\in \mathbb{R}^2$, we have 
        \begin{equation}
        \label{eq:approxVConverges}
            \adjustlimits\lim_{\varepsilon\to 0}\sup_{T\in[0,T_0]}(T+\eps^2)^{1/2}|v_\varepsilon(T,X)-\tilde{v}_\varepsilon(T,X)|_{L^2(\mathbb{P})} = 0.
        \end{equation}
    \end{lemma}
    \begin{proof}
        Let $\sigma_\eps$ and $\tilde{\sigma_\eps}$ be the unique solutions to 
        \begin{equation}
             \begin{cases}
            \dot{\sigma}_\eps(t) = -\frac{1}{\log\varepsilon^{-1}}\mathbb{E}[f'(t+\eps^2,\sigma_\eps(t)G^\mathfrak{m}_{t}*\eta_\eps(x))]\sigma_\eps(t);\\
            \sigma_\eps(0) = 1
        \end{cases}
        \end{equation}
        and
         \begin{equation}
             \begin{cases}
            \dot{\tilde{\sigma}}_\varepsilon(t) = -\frac{1}{\log\varepsilon^{-1}}\mathbb{E}[\tilde{f}_{\varepsilon}'(t+\eps^2,\tilde{\sigma}_\varepsilon(t)G_{t}^\mathfrak{m}*\eta_\eps(x))]\tilde{\sigma}_\varepsilon(t);\\
            \sigma_\eps(0) = 1,
        \end{cases}
        \end{equation}
        respectively. Then we have
        \begin{equation*}
            \begin{split}
                |\sigma_\eps(T)-\tilde{\sigma}_\varepsilon(T)| &\leq \frac{1}{\log\varepsilon^{-1}}\int_0^T \mathbb{E}[|f'(s+\eps^2,\sigma_\eps(s)G^{\mathfrak{m}}_{s}*\eta_\eps(x))|]|\sigma_\eps(s) - \tilde{\sigma}_\varepsilon(s)|\,ds\\
                &\qquad + \frac{1}{\log\varepsilon^{-1}}\int_0^T\mathbb{E}[|f'(s+\eps^2,\sigma_\eps(s)G^{\mathfrak{m}}_{s}*\eta_\eps(x))
                \\&\hspace{15em}-\tilde{f}_{\eps}'(s+\eps^2,\tilde{\sigma}_\varepsilon(s)G^{\mathfrak{m}}_{s}*\eta_\eps(x))|]|\tilde{\sigma}_\varepsilon(s)|\,ds\\
                &\leq \frac{C_1}{\log\varepsilon^{-1}}\int_0^T \frac{|\sigma_\eps(s)-\tilde{\sigma}_\varepsilon(s)|}{s+\varepsilon^2}\,ds\\
                &\qquad +\frac{e^{3L_1}}{\log\varepsilon^{-1}}\int_0^T\mathbb{E}[|\tilde{f}_{\eps}'(s+\eps^2,\sigma_\eps(s)G^{\mathfrak{m}}_{s}*\eta_\eps(x))-\tilde{f}_{\eps}'(s+\eps^2,\tilde{\sigma}_\varepsilon(s)G^{\mathfrak{m}}_{s}*\eta_\eps(x))|]\,ds\\
                &\qquad+ \frac{e^{3L_1}}{\log\varepsilon^{-1}}\int_0^T\mathbb{E}[|f'(s+\eps^2,\sigma_\eps(s)G^{\mathfrak{m}}_{s}*\eta_\eps(x))-\tilde{f}_{\eps}'(s+\eps^2,\sigma_\eps(s)G^{\mathfrak{m}}_{s}*\eta_\eps(x))|]\,ds,
            \end{split}
        \end{equation*}
        where $C_1$ is a constant depending only on $\mathfrak{m}$, $\ell_1$, $L_1$,$\gamma_1$, and $T_0$ and we applied (\ref{alphaEpsilonBound}) in the second inequality. Next, using an argument similar to (\ref{contractingItrBd1}--\ref{contractingItrBd2}), we obtain
        \begin{align*}
                |\sigma_\eps(T)-\tilde{\sigma}_\varepsilon(T)|
                &\leq \frac{C_1+C_2g(\varepsilon)^{\gamma_2}}{\log\varepsilon^{-1}}\int_0^T \frac{|\sigma_\eps(s)-\tilde{\sigma}_\varepsilon(s)|}{s+\varepsilon^2}\,ds\\
                &\qquad + \frac{C_3}{\log\varepsilon^{-1}}\int_0^T \frac{\mathbb{P}\left(\tilde{v}_\varepsilon(s,x) \geq \frac{g(\varepsilon)}{\sqrt{s+\varepsilon^2}}\right)^{1/2}}{s+\varepsilon^2}\,ds\\
                &\leq   6C_3 \exp\left\{-\frac{g(\varepsilon)^2}{4e^{{2|\mathfrak{m}|T_0}+6L_1}}\right\}+ \frac{C_1+C_2g(\varepsilon)^{\gamma_2}}{\log\varepsilon^{-1}}\int_0^T \frac{|\sigma_\eps(s)-\tilde{\sigma}_\varepsilon(s)|}{s+\varepsilon^2}\,ds,
        \end{align*}
        where $C_2$ depends only on $\mathfrak{m}$, $\ell_1$, $\ell_2$, $L_1$, $L_2$, $\gamma_1$, and $\gamma_2$, and $C_3$ depends only on $\mathfrak{m}$, $\ell_1$, $L_1$, and $\gamma_1$.
        By Gr\"onwall's inequality, this gives
        \begin{equation}
            |\sigma_\eps(T)-\tilde{\sigma}_\varepsilon(T)| \leq 6C_3\exp\left\{-\frac{g(\varepsilon)^2}{4e^{{2|\mathfrak{m}|T_0}+6L_1}} + 3(C_1+C_2g(\varepsilon)^{\gamma_2})\right\}.
        \end{equation}
        Our choice of $T\in (0,T_0]$ was arbitrary, so taking $\varepsilon\to 0$ and using our assumption that $\gamma_2<2$ gives (\ref{eq:approxVConverges}).
    \end{proof}
\subsection{Convergence to the McKean--Vlasov equation}
    \label{convergenceSection}
    Now that we have established that our approximations in Section~\ref{approximatingFnSection} converge to our corresponding solutions to (\ref{eq:u-f-eqn}) and (\ref{eq:McK-V}), the proof of Theorem~\ref{thm:main-theorem-McK-V} will follow as an immediate consequence of our work in Section~\ref{proof-lipschitz}. Indeed, we have
    \begin{align*}
      \adjustlimits\limsup_{\varepsilon\to 0} \sup_{T\in [0,T_0]}&(T+\eps^2)^{1/2}|u_\varepsilon(T,X)-v_\varepsilon(T,X)|_{L^2(\mathbb{P})}\\ &\leq \limsup\limits_{\varepsilon\to 0}  \sup_{T\in [0,T_0]}(T+\eps^2)^{1/2}|u_\varepsilon(T,X)-\tilde{u}_\varepsilon(T,X)|_{L^2(\mathbb{P})} 
      \\&\qquad + \adjustlimits\limsup_{\eps\to 0} \sup_{T\in [0,T_0]}(T+\eps^2)^{1/2}|\tilde{u}_\varepsilon(T,X)-\tilde{v}_\varepsilon(T,X)|_{L^2(\mathbb{P})}\\&\qquad+\adjustlimits\limsup_{\eps\to 0}\sup_{T\in [0,T_0]}(T+\eps^2)^{1/2} |\tilde{v}_\varepsilon(T,X)-v_\varepsilon(T,X)|_{L^2(\mathbb{P})}\\
                             &=     \adjustlimits\limsup_{\varepsilon\to 0}\sup_{T\in [0,T_0]}(T+\eps^2)^{1/2}|\tilde{u}_\varepsilon(T,X)-\tilde{v}_\varepsilon(T,X)|_{L^2(\mathbb{P})}
    \end{align*}by Lemmas~\ref{approxConverges} and~\ref{approxVConverges},
    so all that remains is to show that \[ \adjustlimits\lim_{\varepsilon\to 0}\sup_{T\in [0,T_0]}(T+\eps^2)^{1/2}|\tilde{u}_\varepsilon(T,X)-\tilde{v}_\varepsilon(T,X)|_{L^2(\mathbb{P})} = 0.\]Recall from Section~\ref{approximatingFnSection} that $\tilde{f}_\eps\in \bar{S}_1(\tilde{L}_1(\eps),\tilde{L}_2(\eps))$, where
    \begin{equation}
         \tilde{L}_1(\varepsilon) = L_1+\ell_1+\ell_1g(\varepsilon)^{\gamma_1}\quad\text{and}\quad \tilde{L}_2(\varepsilon) = L_2+\ell_2+\ell_2g(\varepsilon)^{\gamma_2}.
    \end{equation}
   The result will follow from Proposition \ref{lipSoltnBd} if we choose $g(\varepsilon)$ so that the $C^{1,1}$ constants $\tilde{L}_1$ and $\tilde{L}_2$, which now depend on the parameter $\varepsilon$, are sufficiently small as $\varepsilon \to 0$. To simplify notation, we take $\varepsilon$ sufficiently small so that 
     \begin{equation}
         \tilde{L}_1(\varepsilon) \leq 2\ell_1g(\varepsilon)^{\gamma_1}\quad\text{and}\quad \tilde{L}_2(\varepsilon)\leq 2\ell_2g(\varepsilon)^{\gamma_2}.
    \end{equation}

    Because of our assumption (\ref{gAsymp}) on the asymptotic behavior of $g$, we can take $\eps$ sufficiently small so that the conditions (\ref{eq:delta-eps-condition}) of Proposition \ref{lipSoltnBd} are satisfied. To further simplify notation, we also take $\eps$ sufficiently small so that
    \begin{equation}
        (1-(\tilde{L}_1(\eps)+e^{3\tilde{L}_1(\eps)}\tilde{L}_2(\eps))\delta_{\eps})^{-1}\leq 2.
    \end{equation}
    Then applying Proposition \ref{lipMainEstimate} yields
    \begin{multline}
            \sup_{T\in [0,T_0]}(T+\eps^2)^{1/2}|\tilde{u}_\varepsilon(T,X)-\tilde{v}_\varepsilon(T,X)|_{L^2(\mathbb{P})}\\\leq  \frac{2\ell_2C_{T_0,\mathfrak{m}}g(\eps)^{\gamma_2}}{(\log\varepsilon^{-1})^{1/4}}\exp\left\{36\ell_1g(\varepsilon)^{\gamma_1}+6 e^{12\ell_1g(\varepsilon)^{\gamma_1}+|\mathfrak{m}|T_0}\ell_2g(\varepsilon)^{\gamma_2}\right\}.
        \end{multline}
        Next, taking $\varepsilon$ sufficiently small depending  on $\mathfrak{m}$, $T_0$, $\ell_1$, $\ell_2$, $\gamma_1$, and $\gamma_2$  gives
        \begin{equation}
                2\ell_2C_{T_0,\mathfrak{m}}g(\eps)^{\gamma_2}\exp\left\{36\ell_1g(\varepsilon)^{\gamma_1}+6 e^{12\ell_1g(\varepsilon)^{\gamma_1}+|\mathfrak{m}|T_0}\ell_2g(\varepsilon)^{\gamma_2}\right\}
                \leq \exp\left\{e^{14\ell_1g(\varepsilon)^{\gamma_1}}\right\}.
        \end{equation}
        Now, choosing 
        \begin{equation}
            g(\varepsilon) = \left(\frac{\log\log\log\log(\varepsilon^{-1})}{14\ell_1}\right)^{1/\gamma_1},
        \end{equation}
        which satisfies assumption (\ref{gAsymp}), gives
        \begin{equation}
            \sup_{T\in[0,T_0]}(T+\eps^2)^{1/2}|\tilde{u}_\varepsilon(T,X)-\tilde{v}_\varepsilon(T,X)|_{L^2(\mathbb{P})}\leq \frac{\log\log\varepsilon^{-1}}{(\log\varepsilon^{-1})^{1/4}}.
        \end{equation}
        Taking $\varepsilon\to 0$ completes the proof of Theorem \ref{thm:main-theorem-McK-V}.
\printbibliography

\end{document}